\documentclass[a4paper,11pt]{article}
\usepackage[latin1]{inputenc}
\usepackage[english]{babel}
\usepackage{amsmath}
\usepackage{amsfonts}
\usepackage{amssymb}
\usepackage{epsfig}
\usepackage{amsopn}
\usepackage{amsthm}
\usepackage{color}
\usepackage{graphicx}
\usepackage{subfigure}
\usepackage{enumerate}
\newtheorem{theorem}{Theorem}[section]

\newtheorem{lemma}[theorem]{Lemma}
\newtheorem{proposition}[theorem]{Proposition}
\newtheorem{definition}[theorem]{Definition}

\newtheorem*{theorem*}{Theorem}
\newtheorem*{lemma*}{Lemma}
\newtheorem*{remark*}{Remark}
\newtheorem*{definition*}{Definition}
\newtheorem*{proposition*}{Proposition}
\newtheorem*{corollary*}{Corollary}
\numberwithin{equation}{section}
\newcommand{\real}{\mathbb{R}}

\let\ced=\c         

\def\qed{\,\unskip\kern 6pt \penalty 500
\raise -2pt\hbox{\vrule \vbox to8pt{\hrule width 6pt
\vfill\hrule}\vrule}\par}
\definecolor{darkblue}{rgb}{0.05, .05, .65}
\definecolor{darkgreen}{rgb}{0.1, .65, .1}
\definecolor{darkred}{rgb}{0.8,0,0}
\newcommand{\beqn}{\begin{equation}}
\newcommand{\eeqn}{\end{equation}}
\newcommand{\bear}{\begin{eqnarray}}
\newcommand{\eear}{\end{eqnarray}}
\newcommand{\bean}{\begin{eqnarray*}}
\newcommand{\eean}{\end{eqnarray*}}
\begin{document}

\title{\huge \bf Traveling wave solutions for the generalized Burgers-Fisher equation}

\author{
\Large Razvan Gabriel Iagar\,\footnote{Departamento de Matem\'{a}tica
Aplicada, Ciencia e Ingenieria de los Materiales y Tecnologia
Electr\'onica, Universidad Rey Juan Carlos, M\'{o}stoles,
28933, Madrid, Spain, \textit{e-mail:} razvan.iagar@urjc.es},\\
[4pt] \Large Ariel S\'{a}nchez,\footnote{Departamento de Matem\'{a}tica
Aplicada, Ciencia e Ingenieria de los Materiales y Tecnologia
Electr\'onica, Universidad Rey Juan Carlos, M\'{o}stoles,
28933, Madrid, Spain, \textit{e-mail:} ariel.sanchez@urjc.es}\\
[4pt] }
\date{}
\maketitle

\begin{abstract}
Traveling wave solutions, in the form $u(x,t)=f(x+ct)$, to the generalized Burgers-Fisher equation
$$
\partial_tu=u_{xx}+k(u^n)_x+u^p-u^q, \quad (x,t)\in\real\times(0,\infty),
$$
with $n\geq2$, $p>q\geq1$ and $k>0$, are classified with respect to their speed $c\in(-\infty,\infty)$ and the behavior at $\pm\infty$. The existence and uniqueness of traveling waves with any speed $c\in\real$ is established and their behavior as $x\to\pm\infty$ is described. In particular, it is shown that there exists a unique $c^*\in(0,\infty)$ such that there exists a unique soliton $f^*$ with speed $c^*$ and such that
$$
\lim\limits_{\xi\to-\infty}f^*(\xi)=\lim\limits_{\xi\to\infty}f^*(\xi)=0, \quad \xi=x+ct.
$$
Moreover, if $n<p+q+1$ then $c^*<kn$ and if $n>p+q+1$ then $c^*>kn$. For $c<\min\{c^*,kn\}$, any traveling wave with speed $c$ satisfies
$\lim\limits_{\xi\to-\infty}f(\xi)=0$ and $\lim\limits_{\xi\to\infty}f(\xi)=1$, while for $c>\max\{c^*,kn\}$ any traveling wave with speed $c$ satisfies $\lim\limits_{\xi\to-\infty}f(\xi)=1$ and $\lim\limits_{\xi\to\infty}f(\xi)=0$. In particular, for any speed $c\in(0,c^*)$, there are traveling wave solutions $u$ with speed $c$ such that $u(x,t)\to1$ as $t\to\infty$, in contrast to the non-convective case $k=0$.
\end{abstract}

\


\smallskip


\section{Introduction}

Starting from the seminal papers by Fisher \cite{Fi37} and Kolmogorov, Petrovsky and Piscounoff \cite{KPP37} deriving them from mathematical biology, equations of Fisher-KPP type, that is,
\begin{equation}\label{KPP}
\partial_tu=u_{xx}+u^p-u^q,
\end{equation}
became an important category of partial differential equations, both due to their applications and to their very rich mathematical theory, stemming from the competition between the reaction and the absorption terms (together with the effect of the diffusion). One of the most significant mathematical features of equations such as Eq. \eqref{KPP} is the availability of solutions taking the geometric form of a wave (or front) advancing or retracting with a constant speed $c\in\real$, called \emph{traveling wave solutions}, having the general form
\begin{equation}\label{TWS}
u(x,t)=f(x\pm ct), \quad c>0,
\end{equation}
where $f$ is the profile of the wave. Solutions corresponding to $c=0$ are a particular case and will be called stationary waves in this context.

Since the models proposed in \cite{Fi37} and \cite{KPP37} led to \eqref{KPP} with $p=1$, $q=2$ (or, more general, $q>1$), the name Fisher-KPP refers usually in literature to equations of the form \eqref{KPP} (or generalizations of them) with $1\leq p<q$. Nowadays, this case is rather well understood, new applications have been proposed and, from the mathematical point of view, it has been noticed that the dynamics of the solutions is well represented by solutions in form of traveling waves \eqref{TWS}, see for example the monographs and papers \cite{AW75, GK, SGM94, Ma10, DQZ20} and references therein. In particular, a famous result (see for example the classical papers \cite{Fi37, AW75, SGM94} for a deduction of $c^*$ in different cases) is the existence of a critical speed $c^*>0$ such that traveling waves only exist if the speed is $c\geq c^*$ and do not exist when $c\in(0,c^*)$. The profile of these waves satisfies
$$
\lim\limits_{\xi\to-\infty}f(\xi)=0, \quad \lim\limits_{\xi\to\infty}f(\xi)=1,
$$
and very important asymptotic stability results have been established (for more general equations of the same form): general solutions (under suitable conditions on their initial data) approach, in the limit as $t\to\infty$, geometric patterns similar to these traveling waves, see for example \cite{Biro02, DK12, DGQ20, DQZ20, G20} (and references therein).

The opposite order of exponents, that is, $p>q>0$, has been considered with less interest in literature, due to an initial lack of real world applications. However, general models including Eq. \eqref{KPP} with the latter order of exponents have been more recently proposed in applied sciences, for example, in growth and diffusion models \cite{Ba94} (see also \cite{Ma10}). From the mathematical point of view, it has been observed that stationary waves (that is, $c=0$) are in this case separatrices between either finite time extinction and finite time blow-up if $p>1>q>0$ \cite{HBS14} or between decay as $t\to\infty$ with algebraic time scales and finite time blow-up if $p>q\geq1$ \cite{IS25}; that is, stationary solutions are unstable if $p>q$ and, in fact, solutions with initial conditions lying strictly below them decay in time and vanish (either in finite or infinite time), while solutions stemming from initial conditions lying above these stationary waves blow up in a finite time. Existence and non-existence of traveling waves for the case of degenerate diffusion (that is, $(u^m)_{xx}$ replacing $u_{xx}$ in \eqref{KPP}) with $p>q$ is studied in \cite{SHB05}.

In this paper, we consider the following generalized Burgers-Fisher equation
\begin{equation}\label{eq1}
\partial_tu=u_{xx}+k(u^n)_x+u^p-u^q, \quad (x,t)\in\real\times(0,\infty),
\end{equation}
in the range of exponents
\begin{equation}\label{range.exp}
n\geq2, \quad p>q\geq1, \quad k>0.
\end{equation}
We observe that Eq. \eqref{eq1} differs from Eq. \eqref{KPP} by a convection term of Burgers type, justifying the terminology of generalized Burgers-Fisher equation. Let us stress here that, in fact, our study allows for considering the more general equation
\begin{equation}\label{eq1.gen}
\partial_tv=Av_{xx}+B(v^n)_x+Cv^p-Dv^q,
\end{equation}
posed for $(x,t)\in\real\times(0,\infty)$, with coefficients $A$, $B$, $C$, $D>0$. A simple rescaling
$$
v(x,t)=\lambda u(\mu x,\nu t), \quad \lambda=\left(\frac{D}{C}\right)^{1/(p-q)}, \quad \nu=C^{(1-q)/(p-q)}D^{(p-1)/(p-q)}, \quad \mu=\sqrt{\frac{\nu}{A}}
$$
maps Eq. \eqref{eq1.gen} into Eq. \eqref{eq1}, with
$$
k=\frac{B\lambda^{n-1}\mu}{\nu}=\frac{B}{\sqrt{A}}C^{(q+1-2n)/2(p-q)}D^{-(p+1-2n)/2(p-q)}>0.
$$
Finally, note that our results can be derived as well for convection coefficients $k<0$, since Eq. \eqref{eq1} with $k>0$ is mapped into the same Eq. \eqref{eq1} with a coefficient $-k$ by the simple transformation $x=-y\in\real$.

We are interested in giving a complete classification of traveling wave solutions to Eq. \eqref{eq1}. For simplicity, instead of \eqref{TWS}, we unify the notation by letting
\begin{equation}\label{TW}
u(x,t)=f(\xi)>0, \quad \xi=x+ct\in\real, \quad c\in\real.
\end{equation}
A direct calculation shows that the profile $f$ solves the differential equation
\begin{equation}\label{TWODE}
f''(\xi)=cf'(\xi)-knf^{n-1}(\xi)f'(\xi)-f^p(\xi)+f^q(\xi).
\end{equation}
Let us observe that a first effect of the convection term is a lack of symmetry of \eqref{TWODE} with respect to the similar equation of the traveling wave profiles for Eq. \eqref{KPP}. Indeed, if the convection term $-knf^{n-1}f'$ is removed, the remaining equation is invariant to the change $c\mapsto-c$, $\xi\mapsto-\xi$, which allows to employ easily the usual definition \eqref{TWS} of the traveling waves: for any $c\geq0$, the traveling waves of Eq. \eqref{KPP} come in pairs, one with speed $c$ and one with speed $-c$ and with the same profile. In contrast to this, the convection term destroys this symmetry and thus a definition of traveling waves with speed $c$ ranging in $(-\infty,\infty)$ is easier to handle.

Eq. \eqref{eq1} has been considered recently in a number of works \cite{LH24, PS25, Zhang21, Zhang22}, for different cases of exponents $n$, $p$, $q$ (see also references therein). In all these works, having as common point a technique based on an abelian integral, the main idea is to transform Eq. \eqref{eq1} into an autonomous dynamical system which is seen as a small perturbation of a Hamiltonian system. Since the phase portrait of the Hamiltonian system is well known, the authors therein study how, for small perturbation coefficients, the system bifurcates and thus obtain limit cycles or homoclinic and heteroclinic orbits. However, their study only deals (in most of the cases) with natural exponents and with a number of restrictions on them.

In this paper, we focus on the order of exponents which is a priori less interesting from the point of view of applications, but which happens to be much more interesting from the mathematical point of view. Indeed, as we shall see, in the most standard order $p<q$, the convection term does not bring too much novelty with respect to the existence or structure of the traveling wave solutions, except from changing the critical speed $c^*$, and the classification of the traveling waves follows as in, for example, \cite{MM02, MM03, MaOu21, PPS25}. In stark contrast to this picture, in the range \eqref{range.exp} of exponents, the convection term introduces some rather striking changes. We are in a position to state our main results and explain these changes in detail in the next paragraph.

\bigskip

\noindent \textbf{Main results.} As previously explained, our main goal is to classify the traveling wave solutions \eqref{TW} to Eq. \eqref{eq1}, that is, profiles $f$ solving (in classical sense) the differential equation \eqref{TWODE} with respect to their speed $c$ and their behavior as $\xi\to\pm\infty$. This classification will depend on the sign of the difference $n-p-q-1$, a number which is related to a Hopf bifurcation taking place at the speed $c=kn$, when the parameter $c$ varies in $(-\infty,\infty)$. We first state a general result (independent of the bifurcation) of existence and uniqueness of a special traveling wave.
\begin{theorem}\label{th.1}
Let $n$, $p$, $q$ as in \eqref{range.exp}. Then, there exists a unique parameter $c^*\in(0,\infty)$ such that, when $c=c^*$, there is a unique profile $f$ (where uniqueness holds true up to a space translation) solving \eqref{TWODE} and such that
\begin{equation}\label{behcstar}
\lim\limits_{\xi\to-\infty}f(\xi)=\lim\limits_{\xi\to\infty}f(\xi)=0.
\end{equation}
Moreover,
$$
0<c^*<kn+2\sqrt{p-q}
$$
and $c^*>kn$ if $n>p+q+1$, respectively $c^*<kn$ if $n<p+q+1$.
\end{theorem}
The existence of such a wave is a phenomenon specific to the range $p>q$, in the sense that the standard Fisher-KPP equation does not admit such a solution, having zero as limit at both ends. Indeed, its existence is related to the formation of a homoclinic orbit in an associated dynamical system, which is due to a Hopf bifurcation occurring at speed $c=kn$ in the same system (see Section \ref{sec.syst} below). The reason for such a bifurcation is strongly related to the order relation $p>q$ between the reaction and absorption exponents and it does not appear in the opposite case. Moreover, the number $p+q+1$ is, to the best of our knowledge, new and it comes from the Lyapunov number of the above mentioned Hopf bifurcation. As a simple particular case, for $k=0$ (that is, when the convection is removed), $c^*=0$ and the corresponding stationary wave coincides in fact with the semi-explicit stationary solution obtained by the authors in the recent paper \cite{IS25}. In the proof of Theorem \ref{th.1}, the local behavior near $\pm\infty$ of the wave will be made more precise than in \eqref{behcstar} and the exact rates of decay to zero will be obtained, showing that, interestingly, at one end the tail is algebraic and at the other one the tail is a decreasing exponential.

Our next result provides a full classification of the traveling waves with any velocity $c\in\real$, and depends partially on the sign of the number $n-p-q-1$.
\begin{theorem}\label{th.2}
Let $n$, $p$, $q$ satisfy \eqref{range.exp} and let $c\in\real$. Then:
\begin{enumerate}[(a)]
  \item If $q>1$ and $n<p+q+1$, for any $c\geq kn$, there exists a unique (up to a space translation) traveling wave with speed of propagation equal to $c$. Its profile has the following behavior:
\begin{equation}\label{behcbig}
\lim\limits_{\xi\to-\infty}f(\xi)=1, \quad \lim\limits_{\xi\to\infty}\xi^{1/(q-1)}f(\xi)=\left(\frac{c}{q-1}\right)^{1/(q-1)}.
\end{equation}
Moreover, the profile $f$ is non-increasing if $c\geq kn+2\sqrt{p-q}$ and it presents damped oscillations around the constant equilibrium state $f\equiv1$ as $\xi\to-\infty$ if $c\in[kn,kn+2\sqrt{p-q})$.
\item If $q>1$ and $n<p+q+1$, for any $c\in(c^*,kn)$, there exists a unique periodic traveling wave solution with velocity $c$ and presenting infinitely many oscillations around the constant equilibrium state $f\equiv1$. Moreover, there is a unique traveling wave with velocity $c$ presenting infinitely many non-damped oscillations around the equilibrium state $f\equiv 1$ as $\xi\to-\infty$ and decaying to zero as $\xi\to\infty$, in the same form as in \eqref{behcbig}.
\item If $q>1$ and $n<p+q+1$, for any $c\in(-\infty,c^*)$ there exists a unique (up to a space translation) traveling wave with speed of propagation equal to $c$. Its profile has the following behavior:
\begin{equation}\label{behcsmall}
\lim\limits_{\xi\to-\infty}f(\xi)=0, \quad \lim\limits_{\xi\to\infty}f(\xi)=1,
\end{equation}
and it is non-decreasing if $c\in(-\infty,-2\sqrt{p-q})$.
\item If $q>1$ and $n>p+q+1$, the results are similar as in the previous cases, with the difference that the first item holds true for $c\in(c^*,\infty)$, the second item for $c\in(kn,c^*)$ and the third item for $c\in(-\infty,kn)$.
\item If $q=1$, all the previous results hold true, with the only difference that the local behavior as $\xi\to\infty$ in \eqref{behcbig} is replaced by
\begin{equation}\label{behclarge.bis}
\lim\limits_{\xi\to\infty}e^{\lambda\xi}f(\xi)\in(0,\infty), \quad \lambda=\frac{\sqrt{c^2+4}-c}{2}.
\end{equation}
\end{enumerate}
\end{theorem}
In fact, we also deduce the precise behavior as $\xi\to-\infty$ in \eqref{behcsmall}, but due to the different number of cases depending on $c$, we refrain from giving it in the statement of Theorem \ref{th.2}. This local behavior can be found in the statements of Propositions \ref{prop.P1}, \ref{prop.P1c0} and \ref{prop.P1q1} in Section \ref{sec.syst}.

\medskip

\noindent \textbf{Remark.} The limiting case $n=1$, when the convection term is linear, is mapped by the change of variable
$$
v(y,t)=u(x-kt,t), \quad y=x-kt,
$$
into the non-convective equation \eqref{KPP} in terms of $v$. According to the recent paper \cite{IS25}, the soliton is a stationary solution in terms of $v$ and we deduce that $c^*=k$ after undoing the previous change of variable. Moreover, the traveling wave profiles to Eq. \eqref{KPP} have been classified in \cite[Theorem 1]{SHB05} (with $-c$ instead of $c$ in the notation therein). By undoing again the previous change of variable and thus adding $k$ to the speeds appearing in \cite[Theorem 1]{SHB05}, we infer from the latter reference that the results in Theorem \ref{th.1} and \ref{th.2} (parts (a) and (c)) hold true also for $n=1$, with $c^*=k$.

\medskip

\noindent \textbf{Effects of the convection term.} Let us notice first that the existence of traveling wave solutions with any velocity $c\in\real$ is in contrast to similar results on the traveling wave solutions to the standard Fisher-KPP equation \cite{Fi37, KPP37} or even more general models with degenerate or doubly degenerate diffusion \cite{MM02, MM03}, where the existence of traveling waves is limited to a minimal speed (which has been even explicitly deduced in some cases). This very general existence is due to the effect of both the convection term and the opposite order of exponents $p>q$ in our case. Moreover, let us stress here the existence of a periodic traveling wave, oscillating infinitely many times around the equilibrium value one, which is the novel effect of the convection; indeed, the presence of the convection term produces a Hopf bifurcation and thus a limit cycle stemming from it, which is equivalent, in terms of profiles, to this periodic wave. Finally, another important remark is the fact that there exist (according to the third item in Theorem \ref{th.2}) traveling waves connecting the equilibrium levels $f\equiv0$ as $\xi\to-\infty$ to $f\equiv1$ as $\xi\to\infty$ moving both forward and backward (depending on the sign of their speed $c<c^*$). In particular, since $c^*>0$, for any $c\in(0,c^*)$ there are traveling wave solutions $u(x,t)=f(x+ct)$ such that
$$
\lim\limits_{t\to\infty}u(x,t)=1,
$$
with locally uniform convergence. This fact might be of a great importance with respect to the stability of the traveling wave solutions. This situation is completely due to the convection term, since, for $k=0$, we have $c^*=0$, as established in \cite{IS25}, and the change of direction of advance comes together with the change of monotonicity of the profile.

\medskip

The proof of the main results, including more precise tails as $\xi\to\infty$ or $\xi\to-\infty$ in the cases when such limits are equal to zero, follows from a transformation, introduced in Section \ref{sec.syst}, of the differential equation \eqref{TWODE} to an autonomous dynamical system. This is why, we perform throughout the paper a careful analysis of a phase plane associated to the autonomous dynamical system, being as sharp as possible in estimating the ranges of existence of every type of profiles. A part of the points in the analysis will be only sketched, as at some technical points we borrow ideas from the recent paper \cite{ILS25}. We are now ready to pass to the proof of the main theorems, which will be split into several sections and completed with a general recap of all ranges of $c$ at the end of the paper.

\section{An alternative formulation. Local analysis}\label{sec.syst}

As commented in the Introduction, we transform Eq. \eqref{TWODE} into an autonomous dynamical system by letting
\begin{equation}\label{PSchange}
X(\xi)=f(\xi)>0, \quad Y(\xi)=f'(\xi)\in\real, \quad \xi\in\real.
\end{equation}
Eq. \eqref{TWODE} is thus transformed into the following autonomous first order system of class $C^1$ if $q>1$,
\begin{equation}\label{PSsyst}
\left\{\begin{array}{ll}X'=Y, \\Y'=cY-knX_{+}^{n-1}Y-X_{+}^p+X_{+}^q,\end{array}\right.
\end{equation}
where $X_+=\max\{X,0\}$ represents the positive part. If $q=1$, the term $X_{+}^q$ in the second equation of \eqref{PSsyst} is replaced by $X$. Since we are looking for non-negative traveling waves, we are only interested in working with the region $X\geq0$ of the system, but the above definition with the extension by zero allows to apply results in dynamical systems that require the system to be defined and of class $C^1$ in $\real^2$. The system \eqref{PSsyst} has two finite critical points, namely $P_1=(0,0)$ and $P_2=(1,0)$. We next analyze locally the flow of the trajectories of the system in a neighborhood of these critical points. For the first of them, the local analysis depends strongly on the sign of $c$. Let us assume first that $q>1$ and analyze the critical point $P_1$ when $c\neq0$.
\begin{proposition}\label{prop.P1}
The critical point $P_1$ is a non-hyperbolic critical point.

\medskip

(a) If $c>0$, then in the half-plane $(0,\infty)\times\real$ the critical point $P_1$ has a one-dimensional unstable manifold containing a unique trajectory tangent to the vector $e_2=(1,c)$ and a unique center manifold tangent to the line $\{Y=0\}$ with stable direction of the flow. In terms of traveling wave profiles, the unique trajectory contained in the unstable manifold of $P_1$ contains a one-parameter family of profiles such that
\begin{equation}\label{beh.P1u}
\lim\limits_{\xi\to-\infty}e^{c\xi}f(\xi)=L\in(0,\infty),
\end{equation}
while the unique trajectory contained in the center manifold corresponds to profiles such that
\begin{equation}\label{beh.P1c}
\lim\limits_{\xi\to\infty}\xi^{1/(q-1)}f(\xi)=\left(\frac{c}{q-1}\right)^{1/(q-1)}.
\end{equation}

\medskip

(b) If $c<0$, then in the half-plane $(0,\infty)\times\real$ the critical point $P_1$ has a one-dimensional stable manifold containing a unique trajectory tangent to the vector $e_2=(1,c)$ and a unique center manifold tangent to the line $\{Y=0\}$ with unstable direction of the flow. In terms of traveling wave profiles, the unique trajectory contained in the stable manifold of $P_1$ contains a one-parameter family of profiles satisfying
\begin{equation}\label{beh.P1s}
\lim\limits_{\xi\to\infty}e^{c\xi}f(\xi)=L\in(0,\infty),
\end{equation}
while the unique trajectory contained in the center manifold corresponds to profiles satisfying
\begin{equation}\label{beh.P1cbis}
\lim\limits_{\xi\to-\infty}|\xi|^{1/(q-1)}f(\xi)=\left(\frac{|c|}{q-1}\right)^{1/(q-1)}.
\end{equation}
\end{proposition}
\begin{proof}
(a) Assume $c>0$. We observe that the linearization of the system in a neighborhood of $P_1$ has the matrix
$$
M(P_1)=\left(
         \begin{array}{cc}
           0 & 1 \\
           0 & c \\
         \end{array}
       \right),
$$
with eigenvalues $\lambda_1=0$ and $\lambda_2=c$ and respective eigenvectors $e_1=(1,0)$ and $e_2=(1,c)$. Since $c>0$, it follows from the uniqueness theory in \cite[Theorem 3.2.1]{GH} that there is a unique one-dimensional unstable manifold tangent to the vector $e_2$ and (possibly not unique) one-dimensional center manifolds tangent to the vector $e_1$ (and thus to the line $\{Y=0\}$). On the unstable manifold, we have
\begin{equation}\label{interm1}
\lim\limits_{\xi\to-\infty}\frac{Y(\xi)}{X(\xi)}=\lim\limits_{\xi\to-\infty}\frac{f'(\xi)}{f(\xi)}=c.
\end{equation}
The L'Hopital's rule then entails that
$$
\lim\limits_{\xi\to-\infty}\frac{\ln(f(\xi))}{\xi}=\lim\limits_{\xi\to-\infty}\frac{f'(\xi)}{f(\xi)}=c,
$$
which readily gives the local behavior \eqref{beh.P1u}. We now consider a center manifold, which is tangent to the line $\{Y=0\}$. Following the theory in \cite[Theorem 3, Section 2.5]{Carr}, we look for an approximation of the center manifold in the form
$$
Y=\phi(X):=aX^{\theta}+o(X^{\theta}),
$$
with $a\in\real$ and $\theta>1$ to be determined. Taking $\overline{n}=(a\theta X^{\theta-1},-1)$ to be the inward normal vector to the center manifold, the flow of the system across it (which is exactly the same quantity denoted by $(M\phi)(X)$ as in the general theory \cite[Theorem 3, Section 2.5]{Carr}) is given by the expression (neglecting $o(X^{\theta})$)
$$
(M\phi)(X)=a^2\theta X^{2\theta-1}-caX^{\theta}+knaX^{n-1+\theta}+X^p-X^q,
$$
and we wish to choose $\theta$ and $a$ such that $(M\phi)(X)=o(X^{\theta})$ as $X\to0$. Since $p>q$, $n\geq2$ and $\theta>1$, it follows that
$$
(M\phi)(X)=-caX^{\theta}-X^q+o(X^{\theta})+o(X^q),
$$
and it is immediate to see that $\theta=q$ and $ca=-1$, that is, $a=-1/c$. We thus infer that the center manifold has the first approximation
\begin{equation}\label{cmapprox}
Y=-\frac{1}{c}X^{q}+o(X^q)
\end{equation}
in a neighborhood of the origin, and the flow on the center manifold is given by the equation
\begin{equation}\label{reduc}
X'=-\frac{1}{c}X^{q}+o(X^q)<0,
\end{equation}
that is, any center manifold has a stable direction of the flow (that is, pointing toward the critical point). An application of \cite[Theorem 3.2']{Sij} proves that the center manifold is unique in the half-plane $(0,\infty)\times\real$. In terms of traveling wave profiles, the center manifold \eqref{cmapprox} writes
$$
f'(\xi)=-\frac{1}{c}f^q(\xi)+o(f^q(\xi)),
$$
which gives after obvious manipulations and an integration the local behavior \eqref{beh.P1c}, as claimed.

\medskip

(b) For $c<0$, the calculations are similar, but the directions of the flow are reversed. Indeed, the second eigenvalue is now $\lambda_2=c<0$ and thus we have a unique stable manifold, while a totally similar calculation gives that the center manifold \eqref{cmapprox} points into the positive cone $(0,\infty)^2$ of the phase plane and has an unstable direction of the flow, according to \eqref{reduc}. The local behaviors \eqref{beh.P1s}, respectively \eqref{beh.P1cbis}, follow by similar calculations as above.
\end{proof}
When $q>1$ and $c=0$, the matrix $M(P_1)$ has both eigenvalues equal to zero and the analysis is much more involved. We state below the outcome of this analysis and postpone the proof to the end of the paper.
\begin{proposition}\label{prop.P1c0}
When $c=0$, there is a unique trajectory of the system \eqref{PSsyst} going out of the critical point $P_1$ into the positive region $(0,\infty)^2$ of the phase plane, and a unique trajectory of the system \eqref{PSsyst} entering the critical point $P_1$ from the region $(0,\infty)\times(-\infty,0)$ of the phase plane. The traveling wave profiles corresponding to these trajectories have the following local behavior:
\begin{enumerate}
  \item If $n>(q+1)/2$, then the unstable trajectory, respectively the stable trajectory, correspond to the local behavior
\begin{equation}\label{beh.P10case1}
\begin{split}
&\lim\limits_{\xi\to-\infty}|\xi|^{2/(q-1)}f(\xi)=\left[\frac{q-1}{2}\sqrt{\frac{2}{q+1}}\right]^{-2/(q-1)}, \quad {\rm respectively}, \\
&\lim\limits_{\xi\to\infty}\xi^{2/(q-1)}f(\xi)=\left[\frac{q-1}{2}\sqrt{\frac{2}{q+1}}\right]^{-2/(q-1)}.
\end{split}
\end{equation}
  \item If $n=(q+1)/2$, let us introduce
\begin{equation}\label{V12}
v_1:=\frac{\sqrt{k^2n^2+4n}-kn}{2n}, \quad v_2:=\frac{-\sqrt{k^2n^2+4n}-kn}{2n}.
\end{equation}
Then the unstable trajectory, respectively the stable trajectory, correspond to the local behavior
\begin{equation}\label{beh.P10case2}
\begin{split}
&\lim\limits_{\xi\to-\infty}|\xi|^{1/(n-1)}f(\xi)=[v_1(n-1)]^{-1/(n-1)}, \quad {\rm respectively},\\
&\lim\limits_{\xi\to\infty}\xi^{1/(n-1)}f(\xi)=[-v_2(n-1)]^{-1/(n-1)}.
\end{split}
\end{equation}
  \item If $2\leq n<(q+1)/2$, then the unstable trajectory, respectively the stable trajectory, correspond to the local behavior
\begin{equation}\label{beh.P10case3}
\begin{split}
&\lim\limits_{\xi\to-\infty}|\xi|^{1/(q-n)}f(\xi)=\left[\frac{q-n}{kn}\right]^{-1/(q-n)}, \quad {\rm respectively},\\
&\lim\limits_{\xi\to-\infty}\xi^{1/(n-1)}f(\xi)=\left[k(n-1)\right]^{-1/(n-1)}.
\end{split}
\end{equation}
\end{enumerate}
\end{proposition}
Note that, in the case $n=(q+1)/2$, all the three exponents in the local behavior coincide; that is, 
$$
\frac{2}{q-1}=\frac{1}{n-1}=\frac{1}{q-n}.
$$
The proof is based on a number of further transformations in order to "blow-up" the singularity at $(X,Y)=(0,0)$ and is given in the Appendix.

We are left with the local analysis of the origin when $q=1$, which is given in the next statement.
\begin{proposition}\label{prop.P1q1}
Let $q=1$, $c\in\real$ and set
$$
\lambda_1:=\frac{c+\sqrt{c^2+4}}{2}>0, \quad \lambda_2:=\frac{c-\sqrt{c^2+4}}{2}<0.
$$
Then the critical point $P_1$ is a saddle point. The unique trajectory contained in its unstable manifold, respectively its stable manifold, corresponds to traveling wave profiles with the local behavior:
\begin{equation}\label{beh.P1q1}
\lim\limits_{\xi\to-\infty}e^{-\lambda_1\xi}f(\xi)\in(0,\infty), \quad {\rm respectively} \quad
\lim\limits_{\xi\to\infty}e^{-\lambda_2\xi}f(\xi)\in(0,\infty).
\end{equation}
\end{proposition}
\begin{proof}
With $q=1$, the linearization of the system in a neighborhood of $P_1$ has the matrix
$$
\left(
  \begin{array}{cc}
    0 & 1 \\
    1 & c \\
  \end{array}
\right),
$$
with eigenvalues $\lambda_1>0$ and $\lambda_2<0$ and corresponding eigenvectors $e_1=(1,\lambda_1)$ and $e_2=(1,\lambda_2)$, respectively. We infer from \cite[Theorem 3.2.1]{GH} that the stable and unstable manifold are unique, and by the stable manifold theorem, that they are tangent to the vectors $e_1$, respectively $e_2$. We thus deduce that, on these trajectories, we have
$$
\lim\limits_{\xi\to-\infty}\frac{Y(\xi)}{X(\xi)}=\lambda_1, \quad {\rm respectively} \quad \lim\limits_{\xi\to\infty}\frac{Y(\xi)}{X(\xi)}=\lambda_2,
$$
from where the local behavior \eqref{beh.P1q1} follows readily by undoing the change of variable \eqref{PSchange} and an argument based on the L'Hopital's rule similar to the one employed in the proof of Proposition \ref{prop.P1}.
\end{proof}
This is the only difference in the analysis introduced by the limiting case $q=1$. In the rest of the paper, we assume that the condition \eqref{range.exp} is in force.

The analysis of the critical point $P_2$ is the one allowing for a Hopf bifurcation at $c=kn$.
\begin{proposition}\label{prop.P2}
The critical point $P_2$ is
\begin{itemize}
  \item An unstable node if $c\geq kn+2\sqrt{p-q}$.
  \item An unstable focus if $c\in(kn,kn+2\sqrt{p-q})$.
  \item A stable focus if $c\in(kn-2\sqrt{p-q},kn)$.
  \item A stable node if $c\leq kn-2\sqrt{p-q}$.
\end{itemize}
Moreover, if $n\neq p+q+1$, a Hopf bifurcation occurs at $c=kn$, generating at least one limit cycle in the system \eqref{PSsyst} either for $c<kn$ when $n<p+q+1$, or for $c>kn$ when $n>p+q+1$.
\end{proposition}
\begin{proof}
The linearization of the system \eqref{PSsyst} in a neighborhood of the critical point $P_2$ has the matrix
$$
M(P_2)=\left(
         \begin{array}{cc}
           0 & 1 \\
           q-p & c-kn \\
         \end{array}
       \right),
$$
with eigenvalues
$$
\lambda_{+}=\frac{c-kn+\sqrt{(c-kn)^2-4(p-q)}}{2}, \quad \lambda_{-}=\frac{c-kn-\sqrt{(c-kn)^2-4(p-q)}}{2}.
$$
We thus observe that the two eigenvalues are real numbers if either $c\geq kn+2\sqrt{p-q}$ (and in this case both are positive, leading to an unstable node) or $c\leq kn-2\sqrt{p-q}$ (and in this case both are negative, leading to a stable node), while for $c\in(kn-2\sqrt{p-q},kn+2\sqrt{p-q})$ we have ${\rm Re}\lambda_{+}={\rm Re}\lambda_{-}=(c-kn)/2$, and thus $P_2$ is an unstable focus if $c\in(kn,kn+2\sqrt{p-q})$ and a stable focus if $c\in(kn-2\sqrt{p-q},kn)$, completing the classification. We moreover observe that, for $c=kn$, both eigenvalues $\lambda_+$ and $\lambda_-$ are purely imaginary numbers. We show next that a Hopf bifurcation occurs for $c=kn$, employing the theory given in \cite[Section 4.4]{Pe}. To this end, we translate $P_2$ to the origin by setting $X=\overline{X}+1$. In terms of this new variable, the second equation of the system \eqref{PSsyst} writes
\begin{equation}\label{interm2}
\begin{split}
Y'&=cY-knY(\overline{X}+1)^{n-1}-(\overline{X}+1)^p+(\overline{X}+1)^q\\
&=(c-kn)Y+(q-p)\overline{X}-\frac{p(p-1)}{2}\overline{X}^2+\frac{q(q-1)}{2}\overline{X}^2-kn(n-1)\overline{X}Y\\
&+\left[\frac{q(q-1)(q-2)}{6}-\frac{p(p-1)(p-2)}{6}\right]\overline{X}^3-\frac{kn(n-1)(n-2)}{2}\overline{X}^2Y\\&+o(|(\overline{X},Y)|^3).
\end{split}
\end{equation}
We infer from \eqref{interm2} that the system \eqref{PSsyst} can be written in the general canonical form
$$
\left\{\begin{array}{ll}\overline{X}'=\widetilde{a}\overline{X}+\widetilde{b}Y+\widetilde{p}(\overline{X},Y), \\Y'=\widetilde{c}\overline{X}+\widetilde{d}Y+\widetilde{q}(\overline{X},Y),\end{array}\right.
$$
with $\widetilde{a}=0$, $\widetilde{b}=1$, $\widetilde{c}=q-p<0$, $\widetilde{d}=c-kn$, $\widetilde{p}\equiv0$ and
$$
\widetilde{q}(\overline{X},Y)=\sum\limits_{i+j\geq2}b_{ij}\overline{X}^{i}Y^{j},
$$
with
\begin{equation*}
\begin{split}
&b_{20}=\frac{q(q-1)-p(p-1)}{2}, \quad b_{11}=-kn(n-1), \quad b_{02}=0, \\
&b_{30}=\frac{q(q-1)(q-2)-p(p-1)(p-2)}{6}, \quad b_{21}=-\frac{kn(n-1)(n-2)}{2}, \quad b_{12}=b_{03}=0.
\end{split}
\end{equation*}
We can thus compute the Lyapunov number (cf. \cite[Section 4.4]{Pe}) of the bifurcation occuring at $\widetilde{d}=0$ or $c=kn$ as follows:
\begin{equation*}
\begin{split}
\sigma&=-\frac{3\pi}{2(p-q)^{3/2}}\left[\frac{kn(n-1)[q(q-1)-p(p-1)]}{2}-\frac{(q-p)kn(n-1)(n-2)}{2}\right]\\
&=-\frac{3\pi}{2(p-q)^{3/2}}\frac{kn(n-1)}{2}[q(q-1)-p(p-1)+(p-q)(n-2)]\\
&=-\frac{3kn(n-1)\pi}{4(p-q)^{1/2}}(n-p-q-1).
\end{split}
\end{equation*}
We thus infer that, if $n<p+q+1$, then $\sigma>0$ and thus a \emph{subcritical Hopf bifurcation} occurs at $c=kn$, resulting in the formation of a unique limit cycle in the system \eqref{PSsyst} in a neighborhood of $P_2$ for $c<kn$ in view of Hopf's Theorem (see, for example, \cite[Theorem 1, Section 4.4]{Pe}). On the contrary, if $n>p+q+1$ then $\sigma<0$ and a \emph{supercritical Hopf bifurcation} occurs at $c=kn$, resulting the formation of a unique limit cycle in the system \eqref{PSsyst} in a neighborhood of $P_2$ for $c>kn$, completing the proof.
\end{proof}

\medskip

\noindent \textbf{Remark.} Observe that the previous analysis depends strongly on the fact that $p>q$. In the more common opposite order $p<q$ (the standard Fisher order of exponents), the critical point $P_2$ is a hyperbolic saddle point, there is no bifurcation stemming from it, and the analysis of the phase plane associated to the system \eqref{PSsyst} is rather straightforward. This is why, the case $p>q$ is mathematically much more interesting.

In order to fix the notation, for any $c\in\real$, we denote by $l_1(c)$, respectively $l_0(c)$, the unique unstable trajectory of $P_1$, respectively the unique stable trajectory of $P_1$.

\section{The phase portrait when $c\geq kn+2\sqrt{p-q}$}\label{sec.large}

In this section, we perform a global analysis of the trajectories $l_1(c)$ and $l_0(c)$ for $c\geq kn+2\sqrt{p-q}$. We start with a preparatory calculus lemma.
\begin{lemma}\label{lem.calc}
Let $p>q\geq1$. Then, the function
$$
h:[0,1)\mapsto\real, \quad h(x)=\frac{x^p-x^q}{x-1}
$$
is increasing on $[0,1)$.
\end{lemma}
\begin{proof}
We split the analysis into two cases.

\medskip

\noindent \textbf{Case 1: $0<p-q\leq1$.} A straightforward calculation gives
$$
h'(x)=\frac{x^{q-1}}{(x-1)^2}\overline{l}(x), \quad \overline{l}(x):=(p-1)x^{p-q+1}-px^{p-q}-(q-1)x+q.
$$
We then observe that
$$
\overline{l}'(x)=(p-1)(p-q+1)x^{p-q}-p(p-q)x^{p-q-1}-(q-1)
$$
and
$$
\overline{l}''(x)=(p-1)(p-q+1)(p-q)x^{p-q-1}-p(p-q)(p-q-1)x^{p-q-2}>0, \quad 0<x<1,
$$
hence $\overline{l}'$ is increasing. Thus, $\overline{l}'(x)\leq \overline{l}'(1)=0$ for any $x\in[0,1]$ and thus $\overline{l}$ is non-increasing. This entails that $\overline{l}(x)\geq \overline{l}(1)=0$ for any $x\in[0,1]$, and thus $h'(x)\geq0$ (with equality if and only if $x=0$), completing the proof.

\medskip

\noindent \textbf{Case 2: $p-q>1$.} In this case, we can write
$$
h(x)=x^q\widetilde{h}(x), \quad \widetilde{h}(x):=\frac{x^{p-q}-1}{x-1}.
$$
We deduce by direct calculation that
$$
\widetilde{h}'(x)=\frac{1}{(x-1)^2}\widetilde{l}(x), \quad \widetilde{l}(x):=(p-q-1)x^{p-q}-(p-q)x^{p-q-1}+1.
$$
Then
$$
\widetilde{l}'(x)=(p-q)(p-q-1)(x^{p-q-1}-x^{p-q-2})\leq0, \quad x\in[0,1].
$$
Thus, $\widetilde{l}$ is non-increasing and $\widetilde{l}(x)\geq\widetilde{l}(1)=0$. We infer that $\widetilde{h}'(x)\geq0$, hence $\widetilde{h}$ is increasing and thus $h$ is increasing as well, as claimed.
\end{proof}
This lemma will be very useful for the following result, characterizing the behavior of the trajectories $l_1(c)$ and $l_0(c)$.
\begin{proposition}\label{prop.large}
Let $n$, $p$, $q$, $k$ as in \eqref{range.exp} and let $c\geq kn+2\sqrt{p-q}$. Then:

$\bullet$ the orbit $l_0(c)$ connects $P_2$ to $P_1$ and, outside its endpoints, is fully contained in the quadrant $(X,Y)\in(0,\infty)\times(-\infty,0)$ of the phase plane.

$\bullet$ the orbit $l_1(c)$ crosses the $X$-axis at a point $(X_1(c),0)$ with $X_1(c)>1$.
\end{proposition}
\begin{proof}
In order to prove the first item, we look for a line of equation $Y=m(X-1)$ (and normal vector $\overline{n}=(m,-1)$), where $m>0$ is to be chosen. The direction of the flow of the system \eqref{PSsyst} across the segment of this line corresponding to $X\in[0,1]$ is given by the sign of the expression
\begin{equation}\label{flow.line}
F(X):=(X-1)\left[m^2-mc+mknX^{n-1}+\frac{X^p-X^q}{X-1}\right].
\end{equation}
Lemma \ref{lem.calc} ensures that
$$
\frac{X^p-X^q}{X-1}\leq\lim\limits_{X\to1}\frac{X^p-X^q}{X-1}=p-q, \quad X\in[0,1),
$$
hence, taking into account that $X\leq1$, we deduce from \eqref{flow.line} that
\begin{equation}\label{interm3}
F(X)\geq(X-1)[m^2-m(c-kn)+p-q], \quad X\in[0,1].
\end{equation}
Since $c\geq kn+2\sqrt{p-q}$, we deduce that there exists at least a value of $m>0$ such that $m^2-m(c-kn)+p-q\leq0$. Choosing such a value of $m$, we find from \eqref{interm3} that $F(X)\geq0$ for any $X\in[0,1]$. Moreover, the direction of the flow of the system \eqref{PSsyst} across the line segment $\{Y=0, X\in[0,1]\}$ is given by the sign of $X^q-X^p$, which is positive, and thus points into the first quadrant of the phase plane, while the direction of the flow of the system \eqref{PSsyst} across the negative half-line $\{X=0, Y<0\}$ is given by the sign of $Y$ and thus points into the negative half-space $\{X<0\}$. It follows that the triangle $\mathcal{T}$ limited by the axis $X=0$, $Y=0$ and the line $Y=m(X-1)$ with $m$ chosen as above is negatively invariant and thus the trajectory $l_0(c)$ entering $P_1$ from the interior of $\mathcal{T}$, as proved in Proposition \ref{prop.P1}, remains forever inside this triangle. Since there are no critical points inside $\mathcal{T}$ (and thus no limit cycles as well, by \cite[Theorem 5, Section 3.7]{Pe}), an application of the Poincar\'e-Bendixon's Theorem (see \cite[Section 3.7]{Pe}) ensures that the trajectory $l_0(c)$ comes from the critical point $P_2$, as desired.

\medskip

In order to prove the second item of the proposition, we consider the isocline $Y'=0$, that is, the curve of equation
\begin{equation}\label{isoY}
Y=\frac{X^p-X^q}{c-knX^{n-1}}, \quad X\in\left(0,\left(\frac{c}{kn}\right)^{1/(n-1)}\right),
\end{equation}
which is positive and increasing for $X>1$ and presents a vertical asymptote at $X=(c/kn)^{1/(n-1)}$. We have proved in Proposition \ref{prop.P1} that the trajectory $l_1(c)$ points inside the first quadrant, thus with $X'(\xi)>0$, $Y'(\xi)>0$ in a neighborhood $\xi\in(-\infty,\xi_0)$ for some $\xi_0>0$. Moreover, while it remains inside the first quadrant, we have $X'(\xi)=Y(\xi)>0$ and thus the coordinate $X$ increases. The inverse function theorem then allows us to express the trajectory as a graph $Y=Y(X)$, with derivative
\begin{equation}\label{interm4}
Y'(X)=\frac{cY(X)-knX^{n-1}Y(X)-X^p+X^q}{Y(X)}=c-knX^{n-1}-\frac{X^p-X^q}{Y(X)}.
\end{equation}
Assume for contradiction that $Y(X)$ has a vertical asymptote as $X\to X_0$ for some $X_0\in(0,\infty)$. Then
$$
\lim\limits_{X\to X_0}Y'(X)=c-knX_0^{n-1}\in\real,
$$
leading to a linear growth, which is a contradiction with the vertical asymptote. Thus, the orbit $l_1(c)$ is defined for any $X>0$ and thus crosses the isocline \eqref{isoY}, entering the region where $Y'<0$. Assume now for contradiction that $Y=Y(X)>0$ for any $X\in(0,\infty)$ along the trajectory $l_1(c)$. Since $X'(\xi)=Y(\xi)>0$, it follows that there exists
$$
L=\lim\limits_{X\to\infty}Y(X)\in[0,\infty).
$$
We go back to \eqref{interm4} and pass to the limit as $X\to\infty$ to get $\lim\limits_{X\to\infty}Y'(X)=-\infty$, which leads to an easy contradiction with the horizontal asymptote of the function $Y(X)$. This contradiction, together with the stability of the critical point $P_2$ in the range of $c$ under consideration, shows that there is $X_1(c)\in(1,\infty)$ such that $Y(X_1(c))=0$, completing the proof.
\end{proof}

\section{Monotonicity with respect to $c$}\label{sec.monot}

In this section, we prove that the trajectories $l_1(c)$ and $l_0(c)$ change in a monotone way with respect to the parameter $c$ (the speed of the traveling wave). To this end, let us recall from Propositions \ref{prop.P1}, \ref{prop.P1c0} and \ref{prop.P1q1} that, for any $c\in\real$, the trajectory $l_1(c)$ goes out from $P_1$ into the first quadrant $(0,\infty)^2$ of the phase plane, while the trajectory $l_0(c)$ enters $P_1$ arriving from the quadrant $(0,\infty)\times(-\infty,0)$ of the phase plane. Before intersecting first the axis $Y=0$, the trajectories $l_1(c)$ and $l_0(c)$ are monotone in $X$, more precisely, the $X$-coordinate increases on the trajectory $l_1(c)$ and the $X$-coordinate decreases on the trajectory $l_0(c)$. This fact, together with the inverse function theorem, allows us to express them as graphs of functions $Y=Y_{1,c}(X)$, respectively $Y=Y_{0,c}(X)$, while $Y>0$ (for $l_1(c)$) respectively $Y<0$ (for $l_0(c)$). Let us denote by $X_1(c)$ the first zero of $Y_{1,c}(X)$ (letting by convention $X_1(c)=\infty$ if the trajectory $l_1(c)$ is fully contained in the cone $(0,\infty)^2$). With this notation, we have the following monotonicity property:
\begin{proposition}\label{prop.monot1}
Let $n$, $p$, $q$ and $k$ be as in \eqref{range.exp} and let $c_1<c_2\in\real$. Then $1\leq X_1(c)<\infty$ for any $c\in\real$ and for any $X\in(0,X_1(c_1))$ we have $0<Y_{1,c_1}(X)<Y_{1,c_2}(X)$. Moreover, if $X_1(c_1)>1$, then $X_1(c_1)<X_1(c_2)$.
\end{proposition}
\begin{proof}
We infer from Proposition \ref{prop.P1} that, if $c\neq0$ and $q>1$, the trajectory $l_1(c)$ enters the first quadrant $(0,\infty)^2$ of the phase plane associated to the system \eqref{PSsyst} tangent either to the vector $(1,c)$ for $c>0$ or to the curve (first approximation of the center manifold) $Y=-X^q/c$ if $c<0$, while for $c=0$ the trajectory $l_1(0)$ lies between the trajectory corresponding to $c>0$ and the trajectory corresponding to $c<0$ in any of the three cases stated in Proposition \ref{prop.P1c0}. Thus, the monotonicity with respect to $c$ follows locally in a neighborhood of the critical point $P_1$. The same local monotonicity holds true when $q=1$ according to Proposition \ref{prop.P1q1}.

Let now $c_1<c_2\in\real$. The flow of the system \eqref{PSsyst} with parameter $c_1$ across the trajectory
$$
l(c_2)=(X,Y_{1,c_2}(X)), \quad {\rm with \ normal \ vector} \quad \overline{n}=\left(\frac{dY_{1,c_2}}{dX}(X),-1\right)
$$
is given by the sign of the expression
\begin{equation}\label{interm12}
F(X):=Y_{1,c_2}(X)Y_{1,c_2}'(X)-\left(c_1Y_{1,c_2}(X)-knX^{n-1}Y_{1,c_2}(X)-X^p+X^q\right).
\end{equation}
The inverse function theorem gives
$$
Y_{1,c_2}'(X)=\frac{c_2Y_{1,c_2}(X)-knX^{n-1}Y_{1,c_2}(X)-X^p+X^q}{Y_{1,c_2}(X)}, \quad X\in(0,X_1(c_2)),
$$
and by replacing the previous expression into \eqref{interm12}, we deduce that
$$
F(X)=(c_2-c_1)Y_{1,c_2}(X)>0, \quad X\in(0,X_1(c_2)).
$$
We thus infer that, for any $X\in(0,X_1(c_2))$, we have $Y_{1,c_1}(X)<Y_{1,c_2}(X)$. Since $X_1(c_2)<\infty$ for $c_2$ sufficiently large, as established in Proposition \ref{prop.large}, it follows that $X_1(c_1)\leq X_1(c_2)<\infty$ for any $c_1<c_2$ and thus $X_1(c)<\infty$ for any $c\in\real$. The fact that $X_1(c)\geq1$ is obvious due to the direction of the flow of the system \eqref{PSsyst} across the axis $Y=0$.

We are only left to prove that $X_1(c_1)<X_1(c_2)$ whenever $X_1(c_1)>1$. Assume for contradiction that $X_1(c_1)=X_1(c_2)=:X_0>1$ for some $c_1<c_2\in\real$. Let $\delta\in(0,X_0-1)$. We pick any $X\in(X_0-\delta,X_0)$ and we apply the inverse function theorem to compute
\begin{equation}\label{interm13}
\begin{split}
\left(\frac{dY_{1,c_2}}{dX}-\frac{dY_{1,c_1}}{dX}\right)(X)&=c_2-c_1+\frac{X^q-X^p}{Y_{1,c_2}(X)}-\frac{X^q-X^p}{Y_{1,c_1}(X)}\\
&=c_2-c_1+\frac{(X^q-X^p)(Y_{1,c_1}(X)-Y_{1,c_2}(X))}{Y_{1,c_1}(X)Y_{1,c_2}(X)}>0.
\end{split}
\end{equation}
The positivity in \eqref{interm13} follows, on the one hand, from the fact that $c_2>c_1$ and, on the other hand, from the first part of the proof, which implies that $0<Y_{1,c_1}(X)<Y_{1,c_2}(X)$ if $X\in(0,X_0)$, and the choice of $\delta$ which gives $X>X_0-\delta>1$ and thus $X^q<X^p$. We infer from \eqref{interm13} that the function $X\mapsto Y_{1,c_2}(X)-Y_{1,c_1}(X)$ increases as $X\to X_0$, contradicting that its value is assumed to be equal to zero at $X=X_0$. We have thus proved the strict monotonicity of $X_1(c)$ with respect to $c$, as stated.
\end{proof}
Let us denote by $X_0(c)$ the last zero of $Y_{0,c}(X)$ before reaching the critical point $P_1$ (letting by convention $X_0(c)=\infty$ if the trajectory $l_0(c)$ is fully contained in the cone $(0,\infty)\times(-\infty,0)$). A monotonicity result similar to Proposition \ref{prop.monot1} holds true for the trajectory $l_0(c)$, except for the finiteness of $X_0(c)$.
\begin{proposition}\label{prop.monot0}
Let $n$, $p$, $q$ and $k$ be as in \eqref{range.exp} and let $c_1<c_2\in\real$. Then $1\leq X_0(c)\leq\infty$ for any $c\in\real$ and for any $X\in(0,X_0(c_1))$ we have $Y_{1,c_1}(X)<Y_{1,c_2}(X)<0$. Moreover, if $X_0(c_2)\in(1,\infty)$, then $X_0(c_2)<X_0(c_1)$.
\end{proposition}
The proof is completely similar to the one of Proposition \ref{prop.monot1} and is left to the reader.

\medskip

\noindent \textbf{Remark.} We cannot ensure that $X_0(c)<\infty$ for any $c\in\real$. However, we have proved in Proposition \ref{prop.large} that $X_0(c)=1$ for any $c\in[kn+2\sqrt{p-q},\infty)$, while $X_1(c)>1$ in the same range of $c$.

\section{Phase portrait for $c\in(kn,kn+2\sqrt{p-q})$}\label{sec.nocycle}

In this section, we consider the following range in the analysis of the phase plane, limited from below by the bifurcation speed $c=kn$, under the assumption that $n<p+q+1$. In order to describe with precision the phase portrait in this range, the most difficult point in the analysis is to prove the \emph{non-existence of a limit cycle} inside this range. To this end, we employ a criteria for non-existence of limit cycles for Lienard systems. Before getting into the specific dynamical systems techniques related to this analysis, we prove a preparatory lemma which is the core of the argument.
\begin{lemma}\label{lem.nocycle}
Assume that $n$, $p$, $q$ and $k$ are as in \eqref{range.exp} and that $1<n\leq p+q+1$. Then, if $c\geq kn$, the system of equations
\begin{equation}\label{basic.syst}
\left\{\begin{array}{ll}y^n-x^n=\frac{c}{k}(y-x) \\ \frac{y^{p+1}}{p+1}-\frac{x^{p+1}}{p+1}=\frac{y^{q+1}}{q+1}-\frac{x^{q+1}}{q+1},\end{array}\right.
\end{equation}
has no solution $(x,y)$ with $x\in(0,1)$ and $y\in(1,\infty)$.
\end{lemma}
\begin{proof}
In the first part of the proof, we let exactly $c=kn$ and we show that the system \eqref{basic.syst} has no solutions. The implicit function theorem readily gives that the two equations in the system \eqref{basic.syst} define implicitly two functions $y_1(x)$ and $y_2(x)$ for $x\in[0,1)$, different from the obvious $y_1(x)=y_2(x)=x$, starting with the values obtained by letting $x=0$ in \eqref{basic.syst}
\begin{equation}\label{interm5}
y_1(0)=n^{1/(n-1)}>1, \quad y_2(0)=\left(\frac{p+1}{q+1}\right)^{1/(p-q)}>1
\end{equation}
and such that $y_1(1)=y_2(1)=1$, with a decreasing profile for $x\in(0,1)$. The strategy of the proof is to show that the two functions are always strictly ordered (and thus cannot intersect) for $x\in[0,1)$. We split the proof into several steps, for the easiness of the presentation.

\medskip

\noindent \textbf{Step 1. Order at $x=0$.} We prove in this first step that, if $1<n\leq p+q+1$, then $y_1(0)>y_2(0)$. Recalling \eqref{interm5}, we show first that $y_1(0)$ is a non-increasing function of $n$. Indeed, observing that
$$
y_1(0)=n^{1/(n-1)}=e^{g(n)}, \quad g(x)=\frac{\ln\,x}{x-1},
$$
we deduce by direct calculation that
$$
g'(x)=\frac{1}{(x-1)^2}\left[1-\ln\,x-\frac{1}{x}\right]=\frac{1}{(x-1)^2}\overline{g}(x),
$$
with
$$
\overline{g}'(x)=\frac{1-x}{x^2}<0, \quad x\in(1,\infty),
$$
whence $\overline{g}(x)<\overline{g}(1)=0$. It follows that $g'(x)<0$ for any $x>1$, whence $g$ is decreasing and thus $y_1(0)$ is decreasing as a function of $n$. We infer thus that
\begin{equation}\label{interm6}
y_1(0)\geq(p+q+1)^{1/(p+q)}, \quad n\in(1,p+q+1].
\end{equation}
In view of \eqref{interm6} and the previous arguments, it remains to prove that
\begin{equation}\label{interm7}
p+q+1>\left(\frac{p+1}{q+1}\right)^{(p+q)/(p-q)}, \quad p>q\geq1.
\end{equation}
To this end, we take logarithms in \eqref{interm7} and set $a:=p+q$, $b:=p-q$, thus $a>b>0$. The inequality \eqref{interm7} writes equivalently as
\begin{equation}\label{interm8}
\frac{\ln(a+1)}{a}>\frac{1}{b}\ln\left(1+\frac{2b}{a-b+2}\right),
\end{equation}
and, letting further $d=1/b\in(1/a,\infty)$, we obtain after straightforward manipulations that
$$
\frac{1}{b}\ln\left(1+\frac{2b}{a-b+2}\right)=\frac{1}{a+2}h((a+2)d), \quad h(x)=x\ln\frac{x+1}{x-1}.
$$
It is an easy exercise in basic calculus to prove that $h$ is a non-increasing function in $(1,\infty)$ (we omit the details consisting only in computing the derivative of $h$ up to second order) and, since $x=(a+2)d=(a+2)/b$ and $b\in(0,a)$, we obtain that the right-hand side of \eqref{interm8} is a non-decreasing function of $b\in(0,a)$. Noticing that, for $b=a$, we get an equality in \eqref{interm8}, the inequality \eqref{interm8} (and thus \eqref{interm7}) follows, completing the proof of this step.

\medskip

\noindent \textbf{Step 2. No solution for $n=p+q+1$.} Letting now $n=p+q+1$, and assuming for contradiction that there is a solution $(x,y)$ to \eqref{basic.syst} with $y\neq x$, we sum up the two equations to deduce that
$$
\phi(y)=\phi(x), \quad \phi(t):=\frac{t^{p+q+1}}{p+q+1}-\frac{t^{p+1}}{p+1}+\frac{t^{q+1}}{q+1}-t.
$$
Noticing that
$$
\phi'(t)=(t^p-1)(t^q-1)\geq0, \quad t\in(0,\infty),
$$
it follows that $\phi$ is an injective function and thus $\phi(y)=\phi(x)$ implies $y=x$, and a contradiction with the starting assumption of this step. In particular, we infer from Step 1 and Step 2 that $y_1(x)>y_2(x)$ for $x\in(0,1)$ if $n=p+q+1$.

\medskip

\noindent \textbf{Step 3. Monotonicity with respect to $n$.} Let $n\in(1,p+q+1)$. Since $y_2(x)$ does not depend on $n$, we want to check how $y_1$ changes with $n$, for $x\in(0,1)$, in order to complete the proof of the lemma. We thus compute the derivative of $y_1$:
\begin{equation}\label{interm9}
\begin{split}
y_1'(x)&=\frac{x^{n-1}-1}{y_1^{n-1}(x)-1}=\frac{x^n-x}{y_1^n(x)-y_1(x)}\frac{y_1(x)}{x}\\
&=\frac{y_1(x)^n-n(y_1(x)-x)-x}{y_1^n(x)-y_1(x)}\frac{y_1(x)}{x}\\&=\left[1-\frac{n-1}{y_1(x)^{n-1}-1}\frac{y_1(x)-x}{y_1(x)}\right]\frac{y_1(x)}{x}.
\end{split}
\end{equation}
We notice that, in the previous calculation, the variation with respect to $n$ depends on the mapping
$$
l(t):=\frac{t}{y^t-1}, \quad t=n-1\in[1,p+q), \quad y>1.
$$
Since
$$
l'(t)=\frac{y^t}{(y^t-1)^2}h(t), \quad h(t)=1-t\ln\,t-y^{-t}, \quad h'(t)=(y^{-t}-1)\ln\,y,
$$
we deduce from the fact that $y>1$ that $h$ is a decreasing function, whence $h(t)<h(0)=0$ for $t>0$. Thus, $l'(t)<0$ and thus $l$ is decreasing on $(0,\infty)$. Since $y_1(x)>1>x$ and $y_1'(x)<0$, it follows that the term in brackets in the right-hand side of \eqref{interm9} is negative and thus the right-hand side of \eqref{interm9} (seen as a function of $n$ for independent $y$ and $x$) is decreasing with $n$. Since $y_1(0)$ is decreasing with $n$ as well (as shown in Step 1), a comparison argument in the differential equation \eqref{interm9} shows that $y_1(x)$ is decreasing with $n$ for $x\in(0,1)$. This fact, together with the conclusion of Step 2, ensures that $y_1(x)>y_2(x)$ for any $x\in(0,1)$ and any $n\in(1,p+q+1]$, completing the proof in the limiting case $c=kn$.

\medskip

\noindent \textbf{Step 4. End of the proof.} We consider next the system \eqref{basic.syst} with $c>kn$ and observe again that, if we let $x=0$ in the first equation, we get (with the same notation for $y_1(x)$ and $y_2(x)$) that
$$
y_1(0)=\left(\frac{c}{k}\right)^{1/(n-1)}>n^{1/(n-1)}>(p+q+1)^{1/(p+q)}>y_2(0).
$$
We are thus in a similar situation as at the end of Step 1, and we proceed as in Step 3 by computing the derivative of the implicit function
\begin{equation}\label{interm10}
y_1'(x)=\frac{knx^{n-1}-c}{kny_1^{n-1}(x)-c}, \quad x\in(0,1),
\end{equation}
and observing that the right-hand side is an increasing function with respect to $c$ (for $n$ fixed), since $y_1(x)>1>x$. A comparison argument in the one-parameter family of differential equations \eqref{interm10} (depending on the parameter $c$) entails that $y_1(x)$ is increasing with $c$ as $x\in(0,1)$. Since Step 3 ensures that, for $c=kn$, we already have $y_1(x)>y_2(x)$, and $y_2$ does not depend on $c$, the previous monotonicity implies that the inequality remains true with the implicit function corresponding to any parameter $c>kn$, completing the proof.
\end{proof}
We next observe that Eq. \eqref{TWODE} can be written in the form of a Lienard system by letting
$$
X(\xi)=f(\xi), \quad \overline{Y}(\xi)=f'(\xi)-cf(\xi)+kf^n(\xi),
$$
and obtaining the alternative dynamical system
\begin{equation}\label{Lienardsyst}
\left\{\begin{array}{ll}X'=\overline{Y}+cX-kX^n, \\ \overline{Y}'=X^q-X^p,\end{array}\right.
\end{equation}
and we easily notice that the critical point $P_2$ from the system \eqref{PSsyst} is topologically equivalent to the critical point $P_2'=(1,k-c)$ in the system \eqref{Lienardsyst}. We translate the critical point $P_2'$ to the origin by setting
$$
X=\widetilde{X}+1, \quad \overline{Y}=\widetilde{Y}+k-c,
$$
and the system \eqref{Lienardsyst} is transformed into the system
\begin{equation}\label{Lienardsyst2}
\left\{\begin{array}{ll}\widetilde{X}'=\widetilde{Y}+c\widetilde{X}+k-k(1+\widetilde{X})^n,\\ \widetilde{Y}'=(1+\widetilde{X})^q-(1+\widetilde{X})^p.\end{array}\right.
\end{equation}
We are interested in the range $\widetilde{X}\in(-1,\infty)$. We have the following result:
\begin{lemma}\label{lem.nocycle2}
Let $n$, $p$, $q$ and $k$ as in \eqref{range.exp} such that $n<p+q+1$ and $c\in[kn,kn+2\sqrt{p-q})$. Then the system \eqref{PSsyst} has no limit cycles.
\end{lemma}
\begin{proof}
Introducing the notation
$$
g(\widetilde{X}):=(1+\widetilde{X})^p-(1+\widetilde{X})^q, \quad G(\widetilde{X}):=\frac{1}{p+1}(1+\widetilde{X})^{p+1}-\frac{1}{q+1}(1+\widetilde{X})^{q+1},
$$
respectively
$$
F(\widetilde{X}):=k(1+\widetilde{X})^{n}-k-c\widetilde{X},
$$
we employ the following criterion for non-existence of limit cycles in the system \eqref{Lienardsyst2} (see for example \cite[Theorem 3, Section 3.9]{Pe} and \cite{Han92}): if $ug(u)>0$ for any $u\in(-1,0)\cup(0,\infty)$ and the system of equations
\begin{equation}\label{interm11}
F(u)=F(v), \quad G(u)=G(v)
\end{equation}
has no solution $(u,v)$ such that $u\in(-1,0)$ and $v\in(0,\infty)$, then the system \eqref{Lienardsyst2} has no limit cycles in the region $-1<\widetilde{X}<\infty$. The first condition is easily verified; indeed,
$$
ug(u)=u(1+u)^q((1+u)^p-1)>0, \quad u\in(-1,\infty)\setminus\{0\}.
$$
The system \eqref{interm11} becomes in our case
$$
\left\{\begin{array}{ll}k(1+u)^n-cu=k(1+v)^n-cv, \\ \frac{1}{p+1}(1+u)^{p+1}-\frac{1}{q+1}(1+u)^{p+1}=\frac{1}{p+1}(1+v)^{p+1}-\frac{1}{q+1}(1+v)^{q+1},\end{array}\right.
$$
which, after letting $x:=1+u\in(0,1)$, $y:=1+v\in(1,\infty)$ and easy algebraic manipulations, reduces to the system \eqref{basic.syst}. An application of Lemma \ref{lem.nocycle} ensures that the system \eqref{interm11} has no solutions with $u=x-1\in(-1,0)$, $v=y-1\in(0,\infty)$ and thus the system \eqref{Lienardsyst2} has no limit cycles in the region $\widetilde{X}\in(-1,\infty)$. Since $X=1+\widetilde{X}=f$ in both systems \eqref{Lienardsyst} and \eqref{PSsyst}, it readily follows that \eqref{PSsyst} has no limit cycles for $X\in(0,\infty)$ as well.
\end{proof}
Recalling next the notation $X_1(c)$, respectively $X_0(c)$, introduced in Section \ref{sec.monot}, we can establish the finiteness of $X_0(c)$ in the range of $c$ under consideration.
\begin{lemma}\label{lem.finit}
Let $n$, $p$, $q$ and $k$ as in \eqref{range.exp} such that $n<p+q+1$ and $c\in[kn,kn+2\sqrt{p-q})$. Then $1<X_0(c)<X_1(c)<\infty$.
\end{lemma}
\begin{proof}
The fact that $X_0(c)>1$ follows from the stability of the critical point $P_2$ (an unstable focus, according to Proposition \ref{prop.P2}), which prevents the existence of a direct connection $P_2-P_1$ fully contained in the half-plane $\{Y\leq0\}$. Assume next for contradiction that there is $c_0\in[kn,kn+2\sqrt{p-q})$ such that $X_0(c_0)=\infty$. This implies that $l_0(c_0)$ is fully contained in the half-plane $\{Y\leq0\}$, whence $X'(\xi)<0$ for any $\xi\in(-\infty,\infty)$ and thus there exists
$$
\lim\limits_{\xi\to-\infty}X(\xi)=l\in(1,\infty]
$$
along the trajectory $l_0(c_0)$. If $l<\infty$, the non-existence of finite critical points with $X=l$ and $Y<0$ implies that $Y(\xi)\to-\infty$ as $\xi\to-\infty$, hence the function $Y_{0,c_0}(X)$ has a vertical asymptote at $X=l$. But the inverse function theorem gives that
$$
\frac{dY_{0,c_0}(X)}{dX}=c_0-knX^{n-1}+\frac{X^q-X^p}{Y_{0,c_0}(X)}
$$
and thus
$$
\lim\limits_{X\to l}\frac{dY_{0,c_0}(X)}{dX}=c_0-knl^{n-1}\in\real,
$$
contradicting the vertical asymptote as $X\to l$. We thus deduce that $l=\infty$. Let us now consider the region $\mathcal{R}$ of the half-plane $\{X\geq0\}$ limited by the trajectory $Y=Y_{1,c_0}(X)$ for $X\in(0,X_1(c_0))$, the vertical line $X=X_1(c_0)$ for $Y\in[Y_{0,c_0}(X_1(c_0),0]$ (that is, going from $Y=0$ down to the intersection with the trajectory $l_0(c_0)$) and the trajectory $Y=Y_{0,c_0}(X)$ for $X\in(0,X_1(c_0))$, see Figure \ref{fig0} for a picture of it. The region $\mathcal{R}$ is positively invariant; indeed, the curves $Y=Y_{1,c_0}(X)$ and $Y=Y_{0,c_0}(X)$ are separatrices in the system \eqref{PSsyst} with $c=c_0$, while the direction of the flow of the system \eqref{PSsyst} across the vertical half-line $X=X_{1}(c_0)$ with $Y<0$ points towards the interior of the region $\mathcal{R}$. Let $\xi=\xi_1(c_0)\in\real$ be such that, on the trajectory $l_1(c_0)$,
$$
(X(\xi_1(c_0)),Y(\xi_1(c_0)))=(X_1(c_0),0).
$$
Since $X_1(c_0)>1$, the point $(X_1(c_0),0)$ is not critical, hence the trajectory $l_1(c_0)$ has to cross the $X$-axis and we infer from the first equation of \eqref{PSsyst} that there is $\delta_0>0$ such that, for any $\xi\in(\xi_1(c_0),\xi_1(c_0)+\delta_0)$, we have on the trajectory $l_1(c_0)$ that $(X(\xi),Y(\xi))\in\mathcal{R}$. Since $\mathcal{R}$ is positively invariant, we infer that the trajectory $l_1(c_0)$ remains forever in the compact region $\overline{\mathcal{R}}$. Taking into account the unstable character of the critical point $P_2$ (according to Proposition \ref{prop.P2}), the non-existence of limit cycles (according to Lemma \ref{lem.nocycle2}) and the uniqueness of the trajectory $l_0(c_0)$ entering $P_1$, an application of the Poincar\'e-Bendixon's Theorem \cite[Section 3.7]{Pe} leads to a contradiction. It thus follows that $1<X_0(c)<\infty$ for any $c\in[kn,kn+2\sqrt{p-q})$. Finally, assuming that there is $c_0\in(kn,kn+2\sqrt{p-q})$ such that $X_0(c_0)>X_1(c_0)$, a completely similar argument as the one above, with the same positively invariant region $\mathcal{R}$, leads to a contradiction, completing the proof.
\end{proof}

\begin{figure}[ht!]
  \begin{center}
  \subfigure[Region $\mathcal{R}$ employed in the proof of Lemma \ref{lem.finit}]{\includegraphics[width=7.5cm,height=6cm]{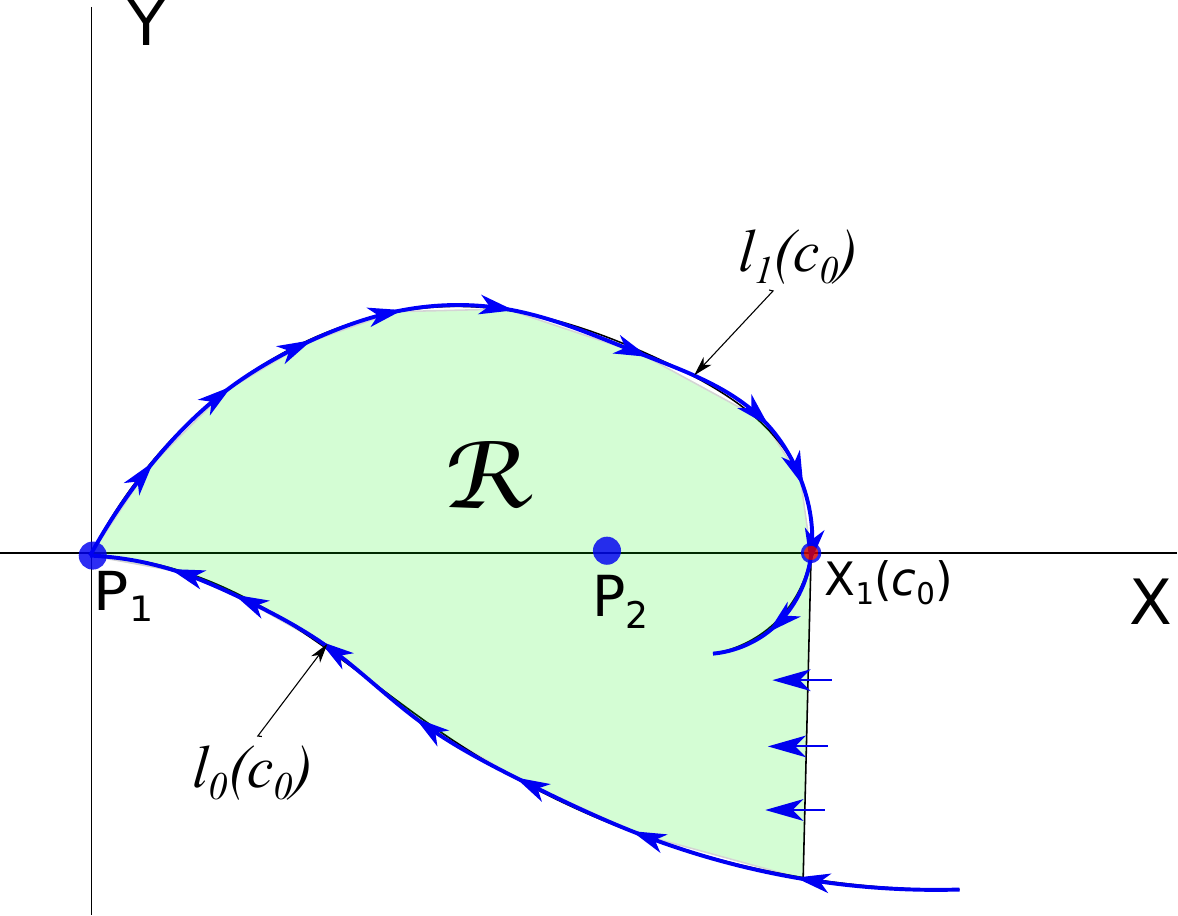}}
  \subfigure[Region $\mathcal{R}'$ employed in the proof of Proposition \ref{prop.sublarge}]{\includegraphics[width=7.5cm,height=6cm]{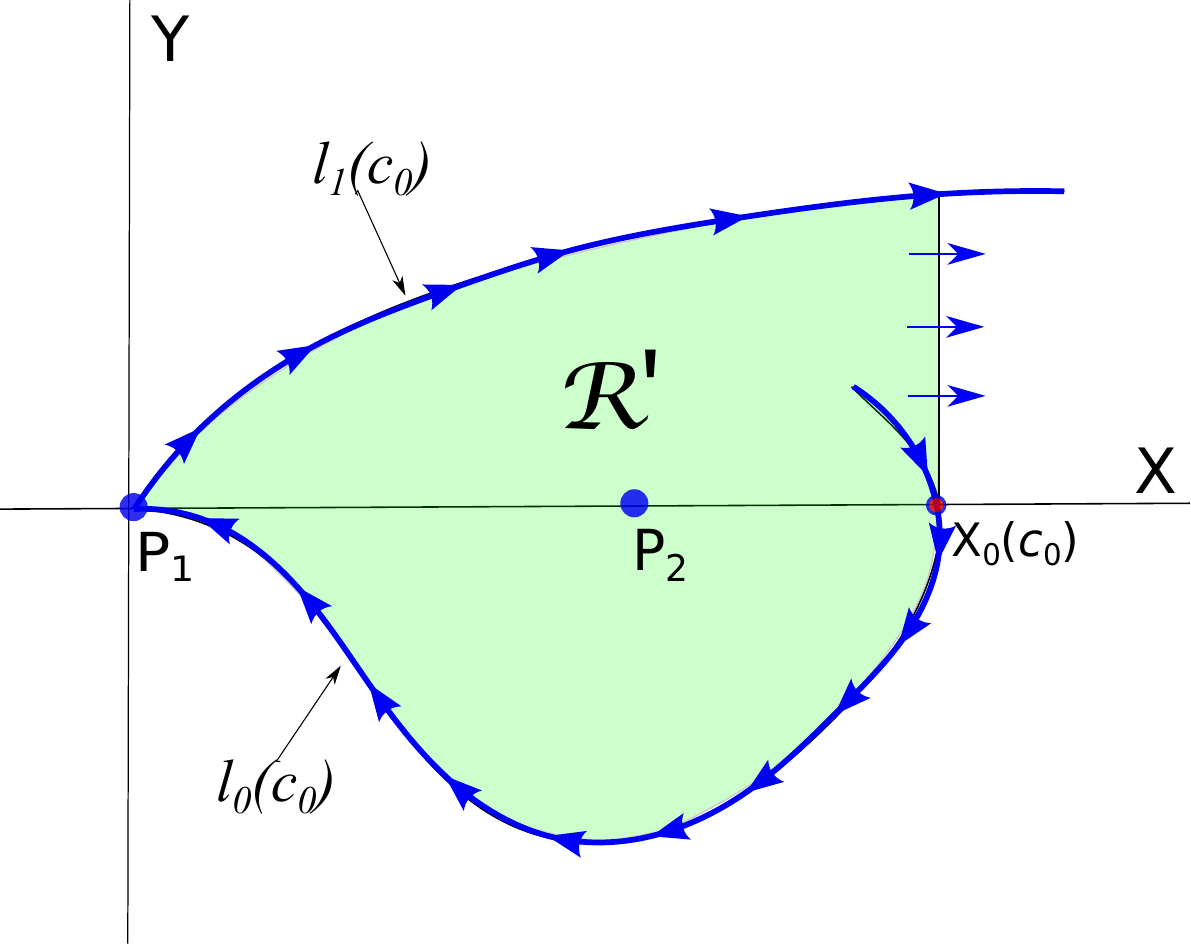}}
  \end{center}
  \caption{Regions $\mathcal{R}$ and $\mathcal{R}'$ used in several proofs.}\label{fig0}
\end{figure} 

With the help of all the previous results, we can complete the phase portrait of the system \eqref{PSsyst} for $c\in(kn,kn+2\sqrt{p-q})$.
\begin{proposition}\label{prop.sublarge}
Let $n$, $p$, $q$ and $k$ as in \eqref{range.exp} such that $n<p+q+1$ and $c\in(kn,kn+2\sqrt{p-q})$. Then the orbit $l_0(c)$ connects the unstable focus $P_2$ to the critical point $P_1$.
\end{proposition}
\begin{proof}
Pick $c\in(kn,kn+2\sqrt{p-q})$. We already know from Lemma \ref{lem.finit} that $1<X_0(c)<X_1(c)<\infty$. We consider the region $\mathcal{R}'$ limited by the trajectory $l_0(c)$, that is, the curve $Y=Y_{0,c}(X)$ for $X\in(0,X_0(c))$, the trajectory $l_1(c)$, that is, the curve $Y=Y_{1,c}(X)$ for $X\in(0,X_0(c))$, and the vertical segment $X=X_0(c)$ for $Y\in[0,Y_{1,c}(X_0(c))]$; see Figure \ref{fig0} for a picture of it. A similar argument as in the previous proof gives that $\mathcal{R}'$ is negatively invariant. Moreover, the first equation of the system \eqref{PSsyst} and the fact that $X_0(c)>1$ implies that the trajectory $l_0(c)$ goes out of the region $\mathcal{R}'$ at $X=X_0(c)$. It follows that the $\alpha-$limit of the trajectory $l_0(c)$ is contained in the compact set $\overline{\mathcal{R}'}$, and once more an application of the Poincar\'e-Bendixon's Theorem, together with the non-existence of limit cycles established in Lemma \ref{lem.nocycle2}, gives that the trajectory $l_0(c)$ starts from the critical point $P_2$, as claimed.
\end{proof}
We plot in Figure \ref{fig1} the phase portrait of the system \eqref{PSsyst} in this range, with the two trajectories $l_0(c)$ (connecting $P_2$ to $P_1$) and $l_1(c)$. 

\begin{figure}[ht!]
  \begin{center}
  \includegraphics[width=11cm,height=7.5cm]{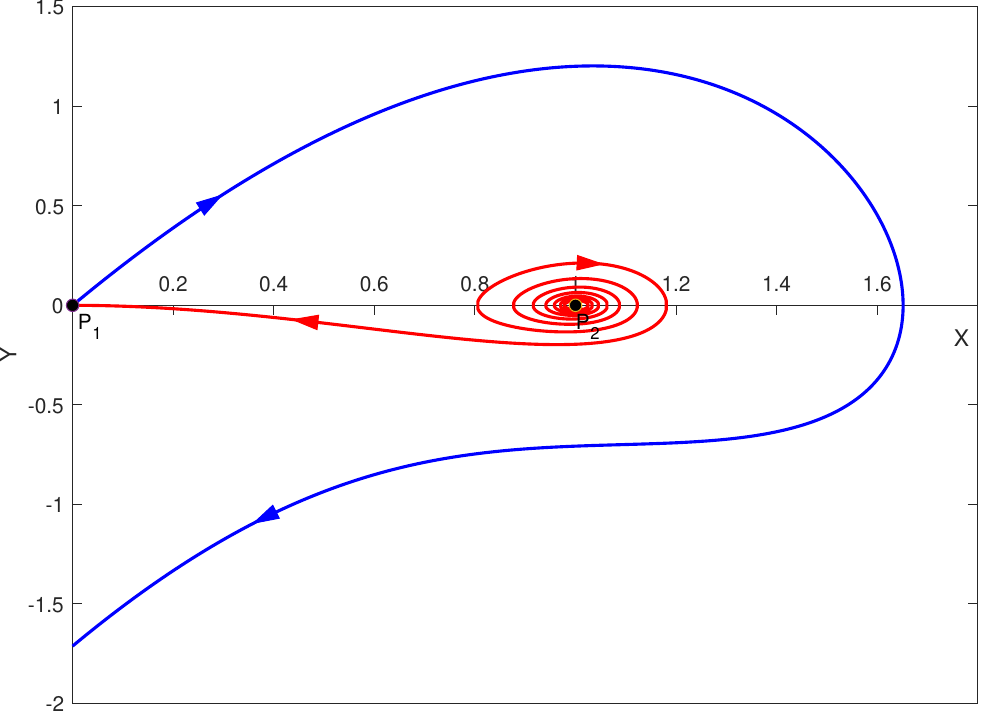}
  \end{center}
  \caption{Illustration of the trajectories $l_1(c)$ and $l_0(c)$ for $c\in(kn,kn+2\sqrt{p-q})$. Experiment for $n=2$, $p=3$, $q=2$, $k=1$ and $c=2.1$.}\label{fig1}
\end{figure}

\section{Phase portrait for $c\in(-\infty,-2\sqrt{p-q})$}\label{sec.small}

In this short section, we establish the configuration of the phase plane associated to the system \eqref{PSsyst} for $c$ in a neighborhood of $-\infty$. More precisely, we have
\begin{proposition}\label{prop.small}
Let $n$, $p$, $q$ and $k$ as in \eqref{range.exp} such that $n<p+q+1$ and $c\in(-\infty,-2\sqrt{p-q})$. Then $X_1(c)=1$; that is, the trajectory $l_1(c)$ connects the critical point $P_1$ to the stable node $P_2$.
\end{proposition}
\begin{proof}
We argue as in the proof of Proposition \ref{prop.large} and consider a line segment of equation $Y=m(1-X)$, $X\in[0,1]$, with $m>0$ to be determined. Taking $\overline{n}=(-m,-1)$ as normal direction to this line segment, we deduce by direct calculation that the flow of the system \eqref{PSsyst} across the line $Y=m(1-X)$ is given by the sign of the expression
\begin{equation}\label{flow.line2}
G(X)=(X-1)\left[m^2+mc-kmnX^{n-1}+\frac{X^p-X^q}{X-1}\right], \quad X\in(0,1).
\end{equation}
Since $X\in(0,1)$, we deduce from Lemma \ref{lem.calc} that
$$
m^2+mc-kmnX^{n-1}+\frac{X^p-X^q}{X-1}\leq m^2+mc+p-q\leq 0,
$$
provided $c\leq-m-(p-q)/m$, whence $G(X)>0$ for $X\in(0,1)$ in the same range of $c$. Since $c\leq-2\sqrt{p-q}$ and
$$
\min\{m+(p-q)/m:m>0\}=2\sqrt{p-q},
$$
there is $m$ such that $c\leq-m-(p-q)/m$, hence $G(X)>0$ with this choice of $m$. It readily follows that the triangular region limited by the two coordinate axis and the line $Y=m(1-X)$ is a positively invariant region and, since the trajectory $l_1(c)$ enters the triangle, an application of the Poincar\'e-Bendixon's Theorem \cite[Section 3.7]{Pe} proves that the trajectory $l_1(c)$ has to enter the stable node $P_2$ while remaining in the positive cone $(0,\infty)^2$ of the phase plane, that is, $X_1(c)=1$, as claimed. The details are completely similar to the end of the proof of Proposition \ref{prop.large}.
\end{proof}
Let us remark that, for $c\in(-\infty,-2\sqrt{p-q})$, we have $X_1(c)=1<X_0(c)$ (where $X_0(c)$ can be finite or infinite), while for $c>kn$ we have shown in Lemma \ref{lem.finit} that $1<X_0(c)<X_1(c)$. We have only one step left towards the existence of a unique homoclinic orbit stemming from $P_1$, that is, the continuity of $X_1(c)$ and $X_0(c)$ as functions of $c$.

\section{The unique homoclinic orbit}\label{sec.cd}

We prove next the continuous dependence with respect to $c$.
\begin{lemma}\label{lem.cd}
Let $n$, $p$, $q$ and $k$ as in \eqref{range.exp} such that $n<p+q+1$. Then the mappings $c\in\real\mapsto X_1(c)\in[1,\infty)$ and $c\in\real\mapsto X_0(c)\in[1,\infty]$ are continuous.
\end{lemma}
\begin{proof}
The proof follows similar ideas as in \cite[Section 4]{ILS25}, and in fact an adaptation of the proofs therein leads to the conclusion of the continuous dependence of the whole trajectories $l_1(c)$ and $l_0(c)$. However, for the sake of completeness, we sketch the proof of the continuity of $X_1(c)$ and $X_0(c)$ with respect to $c$, giving a shorter argument than the one employed in the proof of \cite[Proposition 4.16]{ILS25}. We split the proof into two steps:

\medskip

\noindent \textbf{Step 1.} We prove the continuity from the right of the mapping $c\mapsto X_1(c)$. Pick $c_0\in\real$ and assume for contradiction that
$$
1\leq X_1(c_0)<X_1:=\lim\limits_{c\to c_0^+}X_1(c),
$$
where the existence of $X_1$ is ensured by the monotonicity of $X_1(c)$, according to Proposition \ref{prop.monot1}. Let $L_1(c_0)$ be the unique trajectory of the system \eqref{PSsyst} with parameter $c=c_0$ passing through the (non-critical) point $(X_1,0)$; more precisely, if $\xi_0\in\real$ is such that $(X,Y)(\xi_0)=(X_1,0)$, we focus on the part of the trajectory $(X,Y)(\xi)$ with $\xi<\xi_0$. Since $X_1(c_0)\geq1$ and the trajectory $l_1(c_0)$ is a separatrix for the system \eqref{PSsyst} with $c=c_0$, the direction of the flow of the system \eqref{PSsyst} across the $X$-axis prevents $L_1(c_0)$ to have crossed the $X$-axis before reaching $(X_1,0)$; that is, $Y(\xi)>0$ while $X(\xi)>0$ for any $\xi<\xi_0$. Moreover, the uniqueness of the trajectory $l_1(c_0)$ established in Propositions \ref{prop.P1}, \ref{prop.P1c0} and \ref{prop.P1q1} entails that $L_1(c_0)$ cannot connect to the critical point $P_1$. It thus follows that:

$\bullet$ either there is $\xi_1\in(-\infty,\xi_0)$ such that $X(\xi_1)=0$, $Y(\xi_1)>0$. In this case it follows by a continuity argument that there are $\delta>0$ and $\epsilon>0$ such that any trajectory of the system \eqref{PSsyst} with parameter $c\in(c_0,c_0+\delta)$ and passing through the (non-critical) point $(X,0)$ with $X\in(X_1,X_1+\epsilon)$, has to cross the $Y$-axis at some point with $Y>0$. This is a contradiction, since there is $\delta'\in(0,\delta)$ such that $X_1(c)\in(X_1,X_1+\epsilon)$ for $c\in(c_0,c_0+\delta')\subset(c_0,c_0+\delta)$ and thus the trajectories $l_1(c)$ with $c\in(c_0,c_0+\delta')$ would feature the same property.

$\bullet$ or there is $X_0\geq0$ such that $Y(\xi)\to\infty$ and $X(\xi)\to X_0$ as $\xi\to-\infty$. In this case, by a similar continuity argument we deduce that the trajectories $l_1(c)$ have to cross a fixed, sufficiently high, horizontal line for any $c\in(c_0,c_0+\delta)$ (for some $\delta>0$), while the monotonicity established in Proposition \ref{prop.monot1} establishes that the region covered by $l_1(c)$ for $c\in(c_0,c_0+\delta)$ in the quadrant $(0,\infty)^2$ is uniformly bounded, leading to a contradiction.

\medskip

\noindent \textbf{Step 2.} We prove now the continuity from the left of the mapping $c\mapsto X_1(c)$. Assume thus for contradiction that there is $c_0\in\real$ such that
$$
1<\overline{X}_1:=\lim\limits_{c\to c_0^-}X_1(c)<X_1(c_0).
$$
Let then $\overline{L}_1(c_0)$ be the unique trajectory of the system \eqref{PSsyst} with $c=c_0$ passing through the point $(\overline{X}_1,0)$. Let $\xi_0\in\real$ be such that $(X,Y)(\xi_0)=(\overline{X}_1,0)$ and focus on the part $(X,Y)(\xi)$ of the trajectory $\overline{L}_1(c_0)$ with $\xi\in(-\infty,\xi_0)$. The uniqueness of the trajectory $l_1(c_0)$ stemming from the origin established in Propositions \ref{prop.P1}, \ref{prop.P1c0} and \ref{prop.P1q1}, and the fact that $l_1(c_0)$ is a separatrix for the system \eqref{PSsyst} with parameter $c_0$, entail that the trajectory $\overline{L}_1(c_0)$ has only two possibilities:

$\bullet$ either there is $\xi_1\in(-\infty,\xi_0)$ such that $Y(\xi_1)=0$, $Y(\xi)>0$ for $\xi\in(\xi_1,\xi_0)$, and $X(\xi_1)\in(0,1)$. In this case, the continuity with respect to the parameter $c$ and the definition of $\overline{X}_1$ as a limit from the left readily lead to a contradiction, since this would imply that the trajectories $l_1(c)$ for $c$ in a sufficiently small left-neighborhood of $c_0$ would have crossed the $X$-axis at some positive value of $X$ before reaching the point $(X_1(c),0)$.

$\bullet$ or any point $(X,Y)(\xi)$ on the trajectory $\overline{L}_1(c_0)$ and with $\xi\in(-\infty,\xi_0)$ lies inside the closed region limited by the $X$-axis and the separatrix $l_1(c_0)$. Since the trajectory $\overline{L}_1(c_0)$ cannot come from $P_1$, this case is only possible (by an application of Poincar\'e-Bendixon's Theorem, \cite[Section 3.7]{Pe}) if $\overline{L}_1(c_0)$ directly starts from $P_2$. But this is only possible if $P_2$ is a stable node for $c=c_0$, that is, $c_0\geq kn+2\sqrt{p-q}$. The stability of $P_2$, the definition of $\overline{X}_1$ as limit from the left and the continuity with respect to $c$ ensure that there is $\delta>0$ such that the trajectories passing through $(X_1(c),0)$ in the system \eqref{PSsyst} with parameter $c$, for $c\in(c_0-\delta,c_0)$, also connect to $P_2$. But this is a contradiction with the fact that the unique trajectory passing through $(X_1(c),0)$ in the system \eqref{PSsyst} with parameter $c$ is exactly $l_1(c)$, arriving from $P_1$.

\medskip

The continuity of the mapping $c\mapsto X_0(c)$ from the left and from the right follows by analogous arguments as the ones above, with the only difference that we are in this case considering the part of the trajectory posterior to the crossing point on the $X$-axis. We leave the details to the reader.
\end{proof}
With this continuous dependence, we can complete the proof of the existence and uniqueness of a homoclinic orbit.
\begin{proposition}\label{prop.hom}
Let $n$, $p$, $q$ and $k$ as in \eqref{range.exp} such that $n<p+q+1$. There exists a unique parameter $c^*\in(0,kn)$ such that, for $c=c^*$, the system \eqref{PSsyst} has a unique homoclinic orbit starting and ending at the critical point $P_1$.
\end{proposition}
\begin{proof}
We infer from Lemma \ref{lem.finit} and Proposition \ref{prop.small} that, on the one hand, $X_1(c)>X_0(c)$ for any $c\geq kn$ and, on the other hand, $X_1(c)=1<X_0(c)$ for any $c\leq-2\sqrt{p-q}$. The opposite monotonicity of the functions $X_1(c)$ and $X_0(c)$ established in Propositions \ref{prop.monot0} and \ref{prop.monot1}, their continuity following from Lemma \ref{lem.cd} and an application of Bolzano's Theorem to the continuous function $c\mapsto X_1(c)-X_0(c)$ ensure that there is a unique value $c^*\in(-2\sqrt{p-q},kn)$ such that $X_1(c^*)=X_0(c^*)>1$. Thus, the trajectories $l_0(c^*)$ and $l_1(c^*)$ coincide, forming a homoclinic orbit stemming from the critical point $P_1$.

We are only left to prove that $c^*>0$. Once the existence of a homoclinic orbit established, we argue by contradiction and assume that $c^*\leq0$. We multiply by $f'$ the differential equation \eqref{TWODE} and integrate over $(-\infty,\infty)$ to find
\begin{equation}\label{interm14}
\int_{-\infty}^{\infty}f'(\xi)f''(\xi)d\xi=\int_{-\infty}^{\infty}(c^*-knf^{n-1}(\xi))(f'(\xi))^2\,d\xi-\int_{-\infty}^{\infty}(f^p(\xi)-f^q(\xi))f'(\xi)\,d\xi.
\end{equation}
Since we are dealing with a homoclinic orbit at $P_1=(0,0)$, that is, $X(\xi)\to0$ and $Y(\xi)\to0$ both as $\xi\to-\infty$ and as $\xi\to\infty$, it follows from \eqref{PSchange} that
$$
\lim\limits_{\xi\to\pm\infty}f(\xi)=\lim\limits_{\xi\to\pm\infty}f'(\xi)=0.
$$
Taking into account these limits, we infer from \eqref{interm14} by integration that
$$
\int_{-\infty}^{\infty}(c^*-knf^{n-1}(\xi))(f'(\xi))^2\,d\xi=0,
$$
which is an obvious contradiction since the integrand is negative, by the assumption that $c^*\leq0$. This contradiction implies that $c^*>0$ and completes the proof.
\end{proof}
Observe that Theorem \ref{th.1} follows from Proposition \ref{prop.hom} and the local behavior established in Proposition \ref{prop.P1}, part (a). We plot in Figure \ref{fig2} the unique homoclinic orbit, whose existence is established in this section.

\begin{figure}[ht!]
  \begin{center}
  \includegraphics[width=11cm,height=7.5cm]{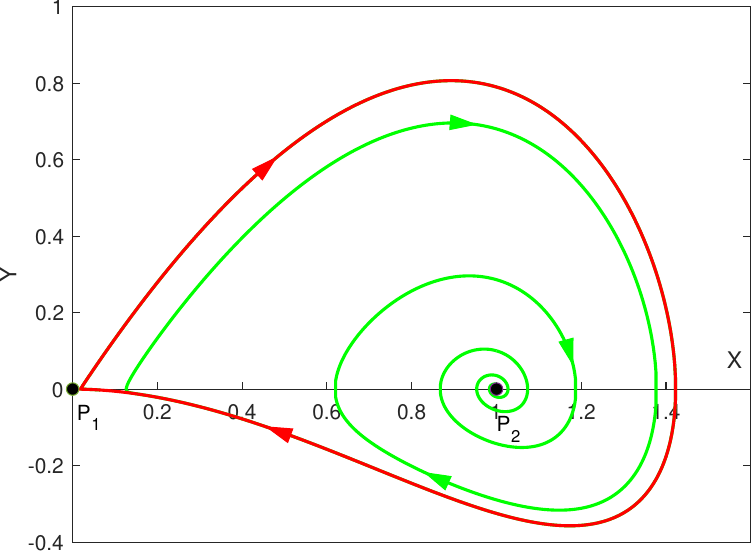}
  \end{center}
  \caption{An approximation of the homoclinic orbit from $P_1$. Experiment for $n=2$, $p=3$, $q=2$, $k=1$ and $c=1.68$.}\label{fig2}
\end{figure}

\section{Uniqueness of the limit cycle and phase portrait for $c<kn$}

Still under the assumption that $n<p+q+1$, we have seen in Proposition \ref{prop.P2} that a unique limit cycle appears locally in a neighborhood of the critical point $P_2$ for $c<kn$, according to Hopf's Theorem \cite[Theorem 1, Section 4.4]{Pe}. We next show that this limit cycle is unique and exists down to $c=c^*$, where $c^*$ is the unique value of the parameter for which a homoclinic orbit at $P_1$ exists, according to Proposition \ref{prop.hom}.
\begin{proposition}\label{prop.cycles}
Let $n$, $p$, $q$ and $k$ as in \eqref{range.exp} such that $n<p+q+1$. Then, for any $c\in(c^*,kn)$, a unique limit cycle exists in the system \eqref{PSsyst} with parameter $c$, and this cycle is unstable. For any $c\in(-\infty,c^*)$, the system \eqref{PSsyst} has no limit cycles.
\end{proposition}
Before giving the proof, let us recall briefly the notion of rotated vector fields according to Han \cite{Han99}, which generalizes previous results by Duff \cite{Duff53} and Perko \cite{Pe75, Pe93}. We thus give the following definition (according to \cite[Definition 1.6]{Han99}):
\begin{definition}\label{def.rvf}
Let $f(x,\lambda)$ be a function defined and analytic on a set $D\times I$, where $D\subseteq\real^2$ is a connected set and $I\subseteq\real$ is an interval. Assume that the critical points lying inside $D$ of the autonomous dynamical system (depending on the parameter $\lambda$)
\begin{equation}\label{dynsyst}
\frac{dx}{dt}=f(x,\lambda)
\end{equation}
remain fixed. We say that the system \eqref{dynsyst} defines a family of \emph{rotated vector fields} with $(x,\lambda)\in D\times I$ if for any $(x_0,\lambda_0)\in D\times I$, with $x_0$ a regular point, there are a neighborhood $D_0\subset D$ of $x_0$ and some $\epsilon>0$ such that $\lambda_0+\epsilon\in I$,
\begin{equation}\label{cond1}
{\rm det}(f(x,\lambda_0),f(x,\lambda))\geq0 \ ({\rm or} \ \leq0), \quad x\in D_0, \quad \lambda\in(\lambda_0,\lambda_0+\epsilon)
\end{equation}
and
\begin{equation}\label{cond2}
{\rm det}(f(x,\lambda_0),f(x,\lambda))\not\equiv0, \quad \lambda\in(\lambda_0,\lambda_0+\epsilon).
\end{equation}
\end{definition}
The rotated vector fields have very strong properties related to the existence, non-existence or uniqueness of limit cycles, according to the results in \cite{Han99}. We are now ready to apply them in order to prove Proposition \ref{prop.cycles}.
\begin{proof}[Proof of Proposition \ref{prop.cycles}]
Let $D=\{(X,Y)\in\real^2: X>0\}$ and $I=\real$ and pick
$$
f(X,Y,c):=(Y,cY-knX^{n-1}Y-X^p+X^q).
$$
We observe that $f$ is analytic on $D\times I$ and ${\rm det}(f(X,Y,c_1),f(X,Y,c_2))=(c_2-c_1)Y^2\geq0$ for any $c_2>c_1\in\real$. Thus, our system \eqref{PSsyst} defines a family of rotated vector fields according to Definition \ref{def.rvf}. Moreover, letting $c_0=kn$, we recall from Proposition \ref{prop.P2} and Hopf's Theorem \cite[Theorem 1, Section 4.4]{Pe} that a (locally unique) unstable limit cycle is generated for $c<kn$ around the critical point $P_2$ and for $c<kn$, the critical point $P_2$ is a stable focus. We thus infer from \cite[Theorem 2.3]{Han99} that the limit cycle generated by the Hopf bifurcation is unique in $D$.

Denoting by $\mathcal{L}(c)$ this unique limit cycle for $c<kn$, we further deduce from \cite[Theorem 2.5]{Han99} that there exists $c_*\in\real$ and some curve (that could be reduced to a single point) $\mathcal{L}^*$ (which is the boundary of the connected region covered by the limit cycles $\mathcal{L}(c)$ for $c\in(c_*,kn)$) such that $\mathcal{L}(c)\to\mathcal{L}^*$ as $c\to c_*$, $c>c_*$, and that the limit $\mathcal{L}^*$ can only be either a single critical point, a singular closed orbit or an unbounded invariant curve. Since the limit cycle is originated from the critical point $P_2$ as $c<kn$, the first of these three alternatives is impossible due to \cite[Theorem 2.2]{Han99}, which states that the cycle moves in a monotone way as $c$ decreases (and thus it cannot ``go back" to extinguish at $P_2$ for some different value of $c$). The third alternative is as well impossible, since it follows from Propositions \ref{prop.monot1} and \ref{prop.monot0} that $\mathcal{L}(c)$ is necessarily contained inside the region limited by the trajectories $l_1(kn)$ for $X\in(0,X_1(kn))$, $l_0(c_*)$ for $X\in(0,X_0(c_*))$ and the vertical line $X=X_1(kn)$ (if $X_1(kn)<X_0(c_*)$) or $X=X_0(c_*)$ (if $X_0(c_*)<X_1(kn)$), which is a bounded closed region. Thus, the only remaining possibility is that $\mathcal{L}^*$ is a singular closed orbit and the uniqueness of a homoclinic orbit proved in Proposition \ref{prop.hom} entails that $c_*=c^*$ and $\mathcal{L}^*$ is the unique homoclinic orbit existing for $c=c^*$.

We are only left to prove the non-existence of limit cycles for $c<c^*$. We first observe that, according to \cite[Theorem 1, Section 3.4]{Pe}, the Poincar\'e map is well defined on the interior side of the homoclinic orbit $\mathcal{L}^*$ for $c=c^*$. Thus, the existence of limit cycles for $c>c^*$ (as explained in the previous paragraph) and \cite[Theorem 2.4 (ii)]{Han99} imply that no limit cycle can exist on the other side of $c^*$, that is, for $c\in(c^*-\epsilon,c^*)$ for some $\epsilon>0$. Moreover, we deduce from Propositions \ref{prop.monot1} and \ref{prop.monot0} that $X_1(c)<X_0(c)$ for $c<c^*$. Assume for contradiction that a limit cycle $\mathcal{L}(\overline{c})$ appears at a parameter $\overline{c}<c^*$. The flow of the system \eqref{PSsyst} across the $X$-axis ensures that any limit cycle surrounding the critical point $P_2$ has the same orientation and thus the non-intersection result \cite[Theorem 2.1]{Han99} readily gives that there exists $c\in(c^*,kn)$ such that $\mathcal{L}(c)=\mathcal{L}(\overline{c})$. But, since both limit cycles $\mathcal{L}(c)$ and $\mathcal{L}(\overline{c})$ are solutions to the system \eqref{PSsyst}, restricting ourselves only to the positive bump (that is, the part of them contained in the positive cone $(0,\infty)^2$) of the two cycles and expressing this part as a graph of a function $Y(X)$, the monotonicity of $Y'(X)$ with respect to the parameter $c$ contradicts the equality of these two identical solutions existing for $\overline{c}<c$. This proves the non-existence of any limit cycle for $c<c^*$.
\end{proof}
We are now in a position to complete the proof of Theorem \ref{th.2} by describing the phase portrait of the system \eqref{PSsyst} in the two remaining cases.
\begin{proposition}\label{prop.upsmall}
For any $c\in(c^*,kn)$, the trajectory $l_0(c)$ connects the unique, unstable limit cycle $\mathcal{L}(c)$ to the critical point $P_1$. For any $c<c^*$, the trajectory $l_1(c)$ connects the critical points $P_1$ and $P_2$.
\end{proposition}
\begin{proof}
Let first $c\in(c^*,kn)$. We know that $1<X_0(c)<X_1(c)<\infty$. Consider thus once more the negatively invariant region $\mathcal{R}'$ plotted in Figure \ref{fig0} and introduced in the proof of Proposition \ref{prop.sublarge}, limited by the trajectory $l_0(c)$, that is, the curve $Y=Y_{0,c}(X)$ for $X\in(0,X_0(c))$, the trajectory $l_1(c)$, that is, the curve $Y=Y_{1,c}(X)$ for $X\in(0,X_0(c))$, and the vertical segment $X=X_0(c)$ for $Y\in[0,Y_{1,c}(X_0(c))]$. We deduce as in the proof of Proposition \ref{prop.sublarge} that the trajectory $l_0(c)$ intersects the $X$-axis at $(X_0(c),0)$ arriving from the interior of the region $\mathcal{R}'$. An application of the Poincar\'e-Bendixon's Theorem then gives that the $\alpha$-limit of $l_0(c)$ is the unstable limit cycle $\mathcal{L}(c)$, as claimed.

Let now $c<c^*$, that is, $X_1(c)<X_0(c)$. We consider again the positively invariant region $\mathcal{R}$ as in the proof of Lemma \ref{lem.finit}, plotted in Figure \ref{fig0} and limited by the trajectory $Y=Y_{1,c}(X)$ for $X\in(0,X_1(c))$, the vertical line $X=X_1(c)$ for $Y\in[Y_{0,c}(X_1(c)),0]$ (that is, going down from $Y=0$ up to the intersection with the trajectory $l_0(c)$) and the trajectory $Y=Y_{0,c}(X)$ for $X\in(0,X_1(c))$. If $X_1(c)>1$, then the trajectory $l_1(c)$ continues after the point $(X_1(c),0)$ by entering the region $\mathcal{R}$. Its positive invariance, the non-existence of limit cycles for $c<c^*$ proved in Proposition \ref{prop.cycles} and an application of the Poincar\'e-Bendixon's Theorem readily imply that the trajectory $l_1(c)$ should end in the stable focus or node $P_2$. If $X_1(c)=1$, the mere definition of $X_1(c)$ ensures that the trajectory $l_1(c)$ connects $P_1$ to $P_2$, completing the proof.
\end{proof}
We are now ready to complete the proof of Theorem \ref{th.2}.
\begin{proof}[Proof of Theorem \ref{th.2}]
We list below a recap of all the results established in the previous propositions, differing with respect to the range of the parameter $c$, assuming that $n<p+q+1$.

$\bullet$ If $c\in[kn,\infty)$, Propositions \ref{prop.large} and \ref{prop.sublarge} give that for any such $c$ there is a unique trajectory of the system \eqref{PSsyst}, that is $l_0(c)$, connecting $P_2$ to $P_0$. In terms of traveling wave profiles, there exists a unique (up to a translation) traveling wave profile with speed $c$ such that
$$
\lim\limits_{\xi\to-\infty}f(\xi)=1, \quad \lim\limits_{\xi\to\infty}f(\xi)=0,
$$
and the local behavior as $\xi\to\infty$ can be made more precise as it follows from Proposition \ref{prop.P1}. Moreover, the traveling wave profile $f$ is decreasing if $c\geq kn+2\sqrt{p-q}$ and has damped oscillations as $\xi\to-\infty$ if $c\in[kn,kn+2\sqrt{p-q})$.

$\bullet$ If $c\in(c^*,kn)$, it follows from Propositions \ref{prop.upsmall} and \ref{prop.cycles} that there is a unique (up to a space translation) periodic traveling wave profile $g$ with speed $c$ oscillating around the value one and there is a unique (up to a space translation) traveling wave profile $f$ with speed $c$ presenting non-damped oscillations as $\xi\to-\infty$ and such that $f(\xi)\to0$ as $\xi\to\infty$, with a more precise local behavior given by Proposition \ref{prop.P1}.

$\bullet$ If $c=c^*$, it follows from Proposition \ref{prop.hom} that there is a unique (up to a space translation) traveling wave profile $f$ with speed $c$ such that
$$
\lim\limits_{\xi\to-\infty}f(\xi)=\lim\limits_{\xi\to\infty}f(\xi)=0,
$$
with a more precise local behavior at both ends following from Proposition \ref{prop.P1}.

$\bullet$ If $c\in(-\infty,c^*)$, it follows from Propositions \ref{prop.upsmall} and \ref{prop.small} that there is a unique (up to a space translation) traveling wave profile $f$ with speed $c$ such that
$$
\lim\limits_{\xi\to-\infty}f(\xi)=0, \quad \lim\limits_{\xi\to\infty}f(\xi)=1.
$$
Moreover, Proposition \ref{prop.small} ensures that the traveling wave profile $f$ is increasing at least if $c\leq-2\sqrt{p-q}$.

By an inspection of the previous proofs, we observe that, if we assume $n>p+q+1$, the only change with respect to the above structure is that in this case $c^*>kn$, the limit cycle bifurcates at $c>kn$ and is a stable limit cycle according to Hopf's Theorem \cite[Theorem 1, Section 4.4]{Pe}, but the previous classification of the traveling wave profiles remains valid (except for the order between $kn$ and $c^*$). Finally, if $n=p+q+1$, then the critical point $P_2$ is a multiple focus and several limit cycles can bifurcate at either $c<kn$ or $c>kn$, but the uniqueness of $c^*$ and the classification of the traveling waves remain in force as well.
\end{proof}
We show in the next two pictures a sample of the phase portrait for $c\in(c^*,kn)$ and for $c\in(0,c^*)$. 

\begin{figure}[ht!]
  \begin{center}
  \subfigure[Typical phase portrait for $c\in(0,c^*)$]{\includegraphics[width=7cm,height=6cm]{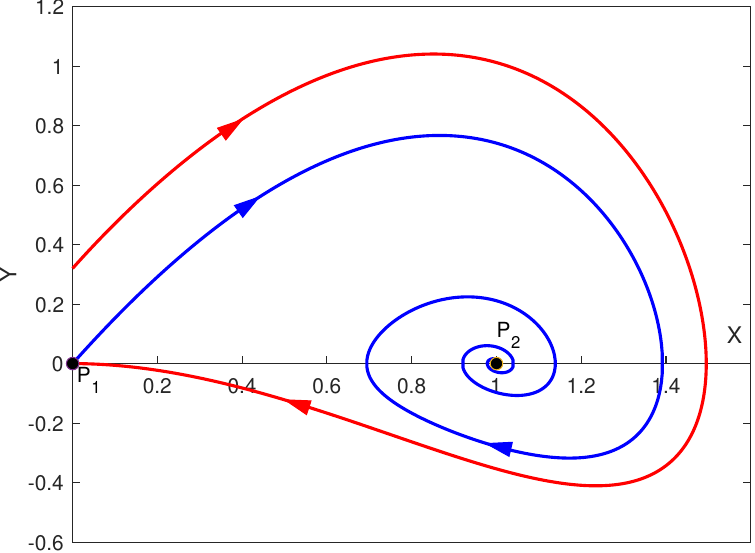}}
  \subfigure[Typical phase portrait for $c\in(c^*,kn)$]{\includegraphics[width=7cm,height=6cm]{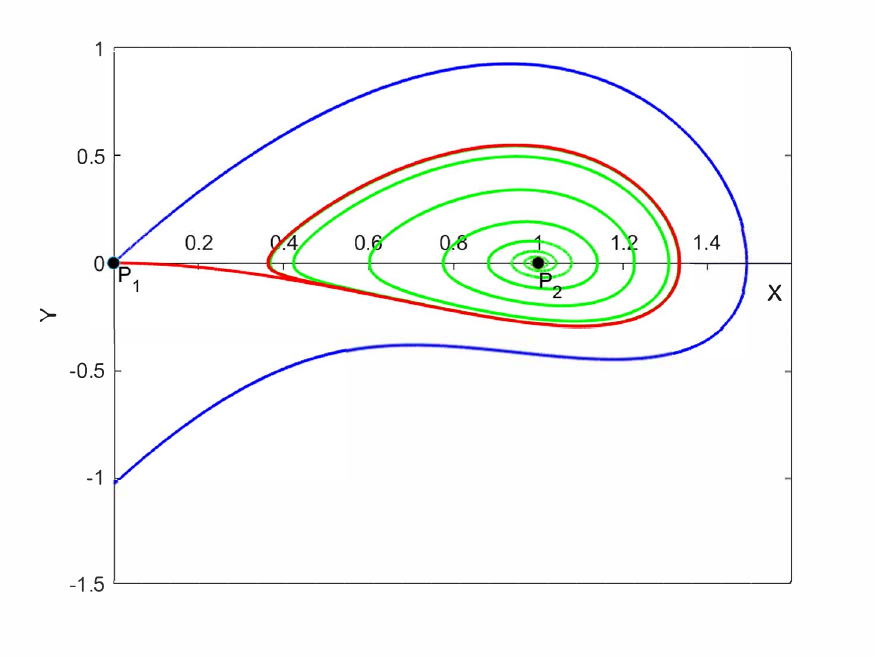}}
  \end{center}
  \caption{Phase portraits with the trajectories $l_1(c)$ and $l_0(c)$ for $c\in(0,c^*)$ (left) and $c\in(c^*,kn)$ (right). Experiments for $n=2$, $p=3$, $q=2$, $k=1$ and $c=1.6$ (left), respectively $c=1.8$ (right).}\label{fig3}
\end{figure}

\section{Appendix. Proof of Proposition \ref{prop.P1c0}}

This final section is devoted to the proof of Proposition \ref{prop.P1c0}, postponed from the main text due to its technical character. Let us thus take $c=0$, and we wish to analyze the origin of the dynamical system
\begin{equation}\label{a1}
\left\{\begin{array}{ll}X'=Y, \\Y'=-knX^{n-1}Y-X^p+X^q,\end{array}\right.
\end{equation}
\begin{proof}[Proof of Proposition \ref{prop.P1c0}]
We notice that the critical point $P_1=(0,0)$ is non-hyperbolic and the linearization of the system in a neighborhood of it has both eigenvalues equal to zero. Thus, new transformations are required in order to reduce the singularity at this point by mapping it into critical points of new, different systems that can be analyzed by standard techniques. Since $(X,Y)\sim(0,0)$, in the second equation of the system \eqref{a1} we have a competition for the dominating order between two terms, $X^q$ and $-knX^{n-1}Y$, and a formal analysis of when one or the other term is dominating leads to the three different cases and transformations that will be considered below.

\medskip

\noindent \textbf{Case 1: $n>(q+1)/2$.} We introduce first the new variable $V=YX^{-(q+1)/2}$. Straightforward calculations starting from the system \eqref{a1} lead to the following system in terms of $(X,V)$
$$
\left\{\begin{array}{ll}X'=X^{(q+1)/2}V, \\V'=-knX^{n-1}V-X^{p-(q+1)/2}+X^{(q-1)/2}-\frac{q+1}{2}V^2X^{(q-1)/2}.\end{array}\right.
$$
After multiplying both equations of the previous system by $X^{(1-q)/q}$ and introducing an implicit change of the independent variable in the form $d\eta/d\xi=X^{(q-1)/q}$, we arrive at the new system
\begin{equation}\label{a2}
\left\{\begin{array}{ll}X'=XV, \\V'=-knX^{n-1+(1-q)/2}V-X^{p-q}+1-\frac{q+1}{2}V^2,\end{array}\right.
\end{equation}
where the derivatives are understood with respect to the new variable $\eta$. We observe that both powers of $X$ in the right hand side of the second equation are positive since $n>(q+1)/2$, but they might not be of class $C^1$. In order to fix this problem and thus be able to apply the standard theory, we replace in \eqref{a2} the variable $X$ by
$$
U:=X^{\mu}, \quad \mu:=\min\left\{p-q,n-1+\frac{1-q}{2}\right\}>0.
$$
Let us consider, for example, that
\begin{equation}\label{a4}
n-1+\frac{1-q}{2}\geq p-q, \quad {\rm that \ is}, \quad n\geq p+\frac{1-q}{2}
\end{equation}
In this case, we have $\mu=p-q$ and we obtain the new system
\begin{equation}\label{a3}
\left\{\begin{array}{ll}U'=(p-q)UV, \\V'=-knU^{(2n-1-q)/(p-q)}V-U+1-\frac{q+1}{2}V^2.\end{array}\right.
\end{equation}
Since $X=0$, it follows that $U=0$ and we are thus looking for the critical points of the system \eqref{a3} with $U=0$. We have two such critical points, namely
$$
Q_1=\left(0,\sqrt{\frac{2}{q+1}}\right), \quad Q_2=\left(0,-\sqrt{\frac{2}{q+1}}\right).
$$
The linearization of the system \eqref{a3} in a neighborhood of the points $Q_1$ and $Q_2$ has the matrix
$$
M(Q_1)=\left(
         \begin{array}{cc}
           (p-q)\sqrt{\frac{2}{q+1}} & 0 \\[1mm]
           -1 & -(q+1)\sqrt{\frac{2}{q+1}} \\
         \end{array}
       \right),
$$
and
$$
M(Q_2)=\left(
         \begin{array}{cc}
           -(p-q)\sqrt{\frac{2}{q+1}} & 0 \\[1mm]
           -1 & (q+1)\sqrt{\frac{2}{q+1}} \\
         \end{array}
       \right),
$$
thus both $Q_1$ and $Q_2$ are saddle points. The stable manifold theorem implies that there is a unique trajectory contained in the unstable manifold of $Q_1$ going out into the first quadrant of the $(U,V)$-plane and a unique trajectory in the stable manifold of $Q_1$ which is contained in the axis $U=0$, due to its invariance for the system \eqref{a3}. Analogously, there is a unique trajectory in the unstable manifold of $Q_2$ which is contained in the invariant axis $U=0$, and a unique trajectory in the stable manifold of $Q_2$ coming from the quadrant $(0,\infty)\times(-\infty,0)$ of the $(U,V)$-plane. The local behavior of the traveling wave profiles corresponding to the trajectory contained in the unstable manifold of $Q_1$ is given by $V(\eta)\to\sqrt{2/(q+1)}$ as $\eta\to-\infty$ or, equivalently, undoing the transformation,
$$
\lim\limits_{\xi\to-\infty}\bigg(YX^{-(q+1)/2}\bigg)(\xi)=\sqrt{\frac{2}{q+1}},
$$
which also reads, after undoing the change of variable \eqref{PSchange},
\begin{equation}\label{a5}
\lim\limits_{\xi\to-\infty}\left(f^{(1-q)/2}\right)'(\xi)=-\frac{q-1}{2}\sqrt{\frac{2}{q+1}}.
\end{equation}
In a similar way, the local behavior of the traveling wave profiles corresponding to the trajectory contained in the stable manifold of $Q_2$ is given by $V(\eta)\to-\sqrt{2/(q+1)}$ as $\eta\to\infty$ or, equivalently, undoing the transformation,
$$
\lim\limits_{\xi\to\infty}\bigg(YX^{-(q+1)/2}\bigg)(\xi)=-\sqrt{\frac{2}{q+1}},
$$
which also reads, after undoing the change of variable \eqref{PSchange},
\begin{equation}\label{a6}
\lim\limits_{\xi\to\infty}\left(f^{(1-q)/2}\right)'(\xi)=\frac{q-1}{2}\sqrt{\frac{2}{q+1}}.
\end{equation}
A simple argument based on the L'Hopital's rule shows that \eqref{a5} and \eqref{a6} imply the local behavior \eqref{beh.P10case1}, as claimed. A completely similar argument is employed when the opposite inequality in \eqref{a4} holds true and thus $\mu=n-1+(1-q)/2$, leading to the same saddle points $Q_1$ and $Q_2$ as above (we omit the details).

\medskip

\noindent \textbf{Case 2: $n=(q+1)/2$.} In this case, we follow similar ideas as in Case 1 by introducing the change of (both dependent and independent) variables
\begin{equation}\label{a12}
U=X^{p-q}, \quad V=X^{-n}Y, \quad \frac{d\eta}{d\xi}=X^{n-1},
\end{equation}
in order to transform the system \eqref{a1} into the following system in the $(U,V)$-variables (where derivatives are taken with respect to $\eta$)
\begin{equation}\label{a13}
\left\{\begin{array}{ll}U'=(p-q)UV, \\V'=-nV^2-knV-U+1.\end{array}\right.
\end{equation}
Since $p>q$, $X=0$ implies $U=0$, and the system \eqref{a13} has two critical points with $U=0$, namely
$$
Q_1=(0,v_1), \quad Q_2=(0,v_2),
$$
where $v_1$ and $v_2$ are defined in \eqref{V12}. The linearization of the system \eqref{a13} in a neighborhood of the critical points $Q_1$ and $Q_2$ has the matrix
$$
M(Q_1)=\left(
         \begin{array}{cc}
           (p-q)v_1 & 0 \\
           -1 & -\sqrt{k^2n^2+4n} \\
         \end{array}
       \right), \quad
M(Q_2)=\left(
         \begin{array}{cc}
           (p-q)v_2 & 0 \\
           -1 & \sqrt{k^2n^2+4n} \\
         \end{array}
       \right).
$$
Since $v_1>0$ and $v_2<0$, we find that $Q_1$ and $Q_2$ are saddle points and, similarly as in Case 1, we are interested only in the unique trajectory contained in the unstable manifold of $Q_1$ and the unique trajectory contained in the stable manifold of $Q_2$. On these trajectories, the local behavior is given by $V(\eta)\to v_1$ as $\eta\to-\infty$, respectively $V(\eta)\to v_2$ as $\eta\to\infty$. By undoing first the change of variable \eqref{a12} and then \eqref{PSchange}, we infer that
\begin{equation}\label{a14}
\lim\limits_{\xi\to-\infty}(f^{1-n})'(\xi)=-(n-1)v_1, \quad {\rm respectively,} \quad \lim\limits_{\xi\to\infty}(f^{1-n})'(\xi)=-(n-1)v_2.
\end{equation}
Starting from \eqref{a14}, a simple argument based on the L'Hopital's rule ensures that \eqref{beh.P10case2} holds true, as stated.
\medskip

\noindent \textbf{Case 3: $n<(1+q)/2$.} In this case, we follow similar ideas as in Case 1 by introducing the change of (both dependent and independent) variables
\begin{equation}\label{a9}
U=X^{q+1-2n}, \quad V=X^{-n}Y, \quad \frac{d\eta}{d\xi}=X^{n-1},
\end{equation}
noting that the assumption $n<(1+q)/2$ gives that $U=0$ when $X=0$. By straightforward calculations, the system \eqref{a1} is mapped into the system (where derivatives are with respect to $\eta$)
\begin{equation}\label{a7}
\left\{\begin{array}{ll}U'=(q+1-2n)UV, \\V'=-nV^2-knV+U-U^{(p+1-2n)/(q+1-2n)}.\end{array}\right.
\end{equation}
We again analyze the critical points with $U=0$, which are
$$
Q_1=(0,0), \quad Q_2=(0,-k).
$$
The linearization of the system in a neighborhood of $Q_1=(0,0)$, respectively $Q_2=(0,-k)$, has the matrix
$$
M(Q_1)=\left(
         \begin{array}{cc}
           0 & 0 \\
           1 & -kn \\
         \end{array}
       \right), \quad
M(Q_2)=\left(
         \begin{array}{cc}
           -k(q+1-2n) & 0 \\
           1 & kn \\
         \end{array}
       \right).
$$
We notice that the critical point $Q_1$ is non-hyperbolic, with one-dimensional center manifolds and a unique (according to \cite[Theorem 3.2.1]{GH}) stable manifold contained in the invariant $U$-axis, while $Q_2$ is a saddle point with a unique stable manifold and a unique unstable manifold contained in the invariant $U$-axis. We are thus interested in the behavior of the profiles corresponding to the center manifold of $Q_1$, respectively the stable manifold of $Q_2$. For the former, an easy application (up to order one) of the center manifold theorem gives that any center manifold of $Q_1$ is approximated by
\begin{equation}\label{a8}
V=\frac{1}{kn}U+o(U),
\end{equation}
whence the flow on any center manifold is given by the reduced first equation
$$
U'=\frac{q+1-2n}{kn}U^2+o(U^2),
$$
hence it is unstable in a neighborhood of $Q_1$. Since the non-zero eigenvalue is negative, the uniqueness theory in \cite[Theorem 3.2 and 3.2']{Sij} applies to prove that the center manifold of $Q_1$ is unique and points into the first quadrant of the $(U,V)$-plane, tangent to the eigenvector $(kn,1)$. In order to deduce the local behavior, we start from \eqref{a8} and undo first the change of variable \eqref{a9} and then \eqref{PSchange} to deduce that
\begin{equation}\label{a10}
\lim\limits_{\xi\to-\infty}\left(f^{n-q}\right)'(\xi)=-\frac{q-n}{kn}.
\end{equation}
For the critical point $Q_2$, we obtain that $V(\eta)\to-k$ as $\eta\to\infty$, that is, $(YX^{-n})(\xi)\to-k$ as $\xi\to\infty$ on the trajectory contained in its stable manifold. By undoing the change of variable \eqref{PSchange}, we are left with
\begin{equation}\label{a11}
\lim\limits_{\xi\to\infty}(f^{1-n})'(\xi)=k(n-1).
\end{equation}
Starting from \eqref{a10} and \eqref{a11}, a simple argument based on the L'Hopital's rule ensures that \eqref{beh.P10case3} holds true, as desired.

\medskip

\noindent \textbf{Uniqueness of the trajectories.} We are only left to prove that, when performing the change from $(X,Y)$ to $(U,V)$ in all the previous cases, we did not miss trajectories connecting to the critical point $P_1=(0,0)$ in the $(X,Y)$-plane but not connecting to one of the resulting points $Q_1$ and $Q_2$ (in some of the three cases) in the $(U,V)$-plane. Since $X=0$ implies $U=0$ in all cases, and there are (in all the three cases under consideration) no other finite critical points of the form $(0,V)$ in the $(U,V)$-plane than the already analyzed $Q_1$ and $Q_2$, the only possibility for such a ``missed" trajectory is to be mapped into a trajectory connecting to a critical point at infinity in the $(U,V)$-plane, more precisely, with $U(\eta)\to0$ and $V(\eta)\to\pm\infty$ as either $\eta\to\infty$ or $\eta\to-\infty$. Assume for contradiction that there exists such a trajectory. Since, in a right neighborhood of $U=0$, in any of the systems \eqref{a3}, \eqref{a7} and \eqref{a13}, the variable $U$ is monotone when $|V|$ is very large, the inverse function theorem gives that this trajectory can be written (at least on some interval $U\in(0,\delta)$) as the graph of a function $V=g(U)$, with derivative
$$
g'(U)=\frac{dV}{dU}.
$$
Let us consider, just as an example, that the trajectory exists in the case $n<(1+q)/2$, that is, in the system \eqref{a7}. We then have, for $U\in(0,\delta)$,
$$
g'(U)=\frac{-ng(U)^2-kng(U)+U-U^{(p+1-2n)/(q+1-2n)}}{(q+1-2n)Ug(U)}.
$$
Letting $U\to0$ and recalling that $g(U)\to\infty$ by assumption, we infer that
$$
g'(U)\sim-\frac{n}{q+1-2n}\frac{g(U)}{U}, \quad {\rm as} \ U\to0.
$$
But the latter equation implies a linear behavior in the first order approximation as $U\to0$, which is a contradiction with the existence of the vertical asymptote. A similar contradiction is obtained also in the systems \eqref{a13} and \eqref{a3}, completing the proof.
\end{proof}

\bigskip

\noindent \textbf{Acknowledgements} This work is partially supported by the Spanish project PID2024-160967NB-I00.

\bigskip

\noindent \textbf{Data availability} Our manuscript has no associated data.

\bigskip

\noindent \textbf{Conflict of interest} The authors declare that there is no conflict of interest.

\bibliographystyle{plain}

\end{document}